\documentclass[a4paper, 10 pt, reqno]{article}
\usepackage{amsfonts, amssymb, amsmath, eucal, amsthm}
\usepackage[all]{xy}
\usepackage{lpic}

\newtheorem{lemma}{Lemma}
\newtheorem{theorem}[lemma]{Theorem}
\newtheorem*{theorem*}{Theorem}
\newtheorem{proposition}[lemma]{Proposition}
\newtheorem{corollary}[lemma]{Corollary}
\newtheorem*{theorema}{Theorem A}
\newtheorem*{theoremb}{Theorem B}
\newtheorem*{theoremc}{Theorem C}
\newtheorem*{corollary*}{Corollary}

\theoremstyle{definition}
\newtheorem{definition}[lemma]{Definition}
\newtheorem{note}[lemma]{Note}
\newtheorem{example}[lemma]{Example}

\newtheorem{notation}[lemma]{Notation}

\numberwithin{lemma}{section}

\newcommand{\C}{\mathcal{C}}
\newcommand{\R}{\mathcal{R}}
\newcommand{\calE}{\mathcal{E}}
\newcommand{\calO}{\mathcal{O}}
\newcommand{\calP}{\mathcal{P}}
\newcommand{\calQ}{\mathcal{Q}}
\newcommand{\calM}{\mathcal{M}}
\newcommand{\calF}{\mathcal{F}}
\newcommand{\rmout}{\mathrm{out}}
\newcommand{\fD}{{{f}\mathcal{D}}}

\newcommand{\Mon}{\mathrm{Mon}}
\newcommand{\CoEnd}{\mathrm{CoEnd}}

\newcommand{\skippy}{\vspace{10 pt}}

\DeclareMathOperator{\Map}{\mathrm{Map}}
\DeclareMathOperator{\Aut}{\mathrm{Aut}}
\DeclareMathOperator{\interior}{\mathrm{int}}

\newcommand{\Getzler}{MR1256989}
\newcommand{\Voronov}{MR2131012}
\newcommand{\CohenVoronov}{MR2240287}
\newcommand{\Kaufmann}{MR2135554}
\newcommand{\James}{MR1010230}
\newcommand{\MenichiBVMorphism}{Menichi}
\newcommand{\KasselTuraev}{MR2435235}
\newcommand{\Salvatore}{MR2482075}
\newcommand{\McClureSmith}{MR2089084}
\newcommand{\ChasSullivan}{ChasSullivan}

\newcommand{\Fied}{Fiedorowicz}
\newcommand{\MayQuasi}{MR1070579}
\newcommand{\Tarje}{Tarje}

\title{Groups, cacti and framed little discs} 
\author{Richard Hepworth \\[10 pt]  {Copenhagen University}}
\date{}

\begin{document}
\maketitle

\begin{abstract}
\noindent 
Let $G$ be a topological group.
Then the based loopspace $\Omega G$ of $G$ is an algebra over the cacti operad, while the double loopspace $\Omega^2 BG$ of the classifying space of $G$ is an algebra over the framed little discs operad.
This paper shows that these two algebras are equivalent, in the sense that they are weakly equivalent $\calE$-algebras, where $\calE$ is an operad weakly equivalent to both framed little discs and cacti.
We recover the equivalence between cacti and framed little discs, and Menichi's isomorphism between the BV-algebras $H_\ast(\Omega G)$ and $H_\ast(\Omega^2 BG)$.
\end{abstract}


\section{Introduction}

The \emph{framed little discs} operad $\fD$ was introduced by Getzler in \cite{\Getzler}.
The double loopspace $\Omega^2 X$ of any based space $X$ is naturally an algebra over $\fD$, and any such algebra is weakly equivalent to one of the form $\Omega^2 BG$ for some topological group $G$ (with non-degenerate basepoint, which we assume throughout).

The \emph{cacti operad} was introduced by Voronov \cite{\Voronov} in order to understand the BV-algebras $\mathbb{H}_{\ast}(LM)$ of Chas and Sullivan \cite{\ChasSullivan}.
We will use a variant of cacti introduced by Salvatore in \cite{\Salvatore}, and will denote it by $\C$.
Salvatore used a version of the Deligne conjecture to show that for any topological group $G$, the loopspace $\Omega G$ is an algebra over $\C$.

So, given a topological group $G$ we have an $\fD$-algebra $\Omega^2 BG$ and a $\C$-algebra $\Omega G$.
There is a standard weak equivalence $\Omega G\simeq\Omega^2 BG$, while $\C$ and $\fD$ are weakly equivalent operads.
It is therefore natural to ask whether $\Omega G$ and $\Omega^2 BG$ are related as \emph{algebras}.
The theorem below shows that this is indeed the case.

\begin{theorem*}
There is an operad $\calE$ equipped with weak equivalences of operads
\[\fD\xleftarrow[\quad\pi_1\quad]{\simeq}\calE\xrightarrow[\quad\pi_2\quad]{\simeq}\C\]
such that, for any topological group $G$, there is an $\calE$-algebra $\varepsilon G$ equipped with weak equivalences of $\calE$-algebras
\[ \Omega^2 BG \xleftarrow[\quad p_1 \quad]{\simeq} \varepsilon G \xrightarrow[\quad p_2 \quad]{\simeq} \Omega G.\]
Here $\Omega^2 BG$ and $\Omega G$ are regarded as $\calE$-algebras using $\pi_1$ and $\pi_2$ respectively, and
$p_1$ and $p_2$ are compatible with the standard weak equivalence $h\colon\Omega G\xrightarrow{\simeq}\Omega^2 BG$ in the sense that $h\circ p_2\simeq p_1$.
\end{theorem*}

The operad $\calE$ is defined as follows.
Given elements $a\in\fD(n)$ and $c\in\C(n)$, write $|a|$ for the space obtained from $D^2$ by deleting the interiors of the little discs of $a$, and write $|c|$ for the configuration of circles underlying $c$.
Then $\calE(n)$ is the space of triples $(a,c,f)$, where $a\in\fD(n)$, $c\in\C(n)$, and where $f\colon |a|\to |c|$ is a homotopy equivalence satisfying the following two boundary conditions.
First, the boundary of the $i$-th little disc is sent into $|c|$ by the inclusion of the $i$-th lobe.
Second, the boundary of the big disc is sent into $|c|$ by the pinch map.
A typical element $(a,c,f)\in\calE(2)$ is depicted below.
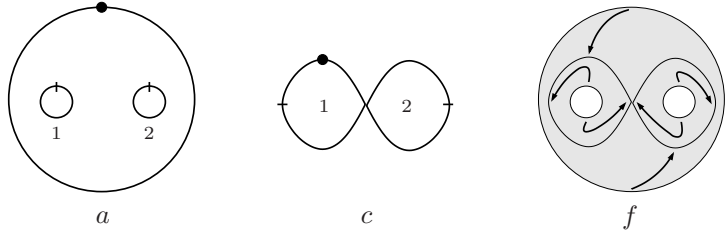
\begin{center}
\begin{lpic}[b(15 pt)]{figures/EElement(,1 in)}
\lbl{38,-10;$a$}
\lbl{143,-10;$c$}
\lbl{248,-10;$f$}
\lbl{19.5,24;$\scriptstyle 1$}
\lbl{56.5,24;$\scriptstyle 2$}
\lbl{126,35;$\scriptstyle 1$}
\lbl{159,35;$\scriptstyle 2$}
\end{lpic}
\end{center}
There are forgetful maps
\[\fD(n)\xleftarrow{\quad\pi_1\quad}\calE(n)\xrightarrow{\quad\pi_2\quad}\C(n).\]
The first part of our theorem can now be restated as follows.
\begin{theorema}
The spaces $\calE=\{\calE(n)\}$ form a topological operad, and $\pi_1$ and $\pi_2$ are weak equivalences of operads.
\end{theorema}

In particular, we recover the fact that cacti and framed little discs are weakly equivalent.
This was stated by Voronov~\cite{\Voronov} and first proved by Kaufmann~\cite{\Kaufmann}.
A different proof has since been given by Bargheer \cite{\Tarje}.
In those accounts, as in Theorem~A, one constructs some intermediate operad $\calO$ equipped with weak equivalences to both $\fD$ and $\C$.
In Kaufmann's proof  $\calO$ is obtained using a recognition principle of Fiedorowicz~\cite{\Fied}, while in Bargheer's proof one takes $\calO=\mathcal{W}\fD$, the Boardman-Vogt $\mathcal{W}$-construction on $\fD$.
The advantage of Theorem~A is that $\calE$ is constructed using only the geometry of the original operads $\fD$ and $\C$.
Further, it allows us to construct a natural $\calE$-algebra $\varepsilon G$ weakly equivalent to both $\Omega G$ and $\Omega^2 BG$.
Note, however, that Theorem~A still relies on Kaufmann's computation of the weak homotopy type of the $\C(n)$.

Now we describe the $\calE$-algebra $\varepsilon G$.
Let $EG\to BG$ be the universal principal $G$-bundle.
Then $\varepsilon G$ is the space of based maps $D^2\to EG$ that send $S^1$ into the orbit of the basepoint.
There are maps
\[\Omega^2 BG \xleftarrow{\quad p_1\quad}\varepsilon G\xrightarrow{\quad p_2 \quad} \Omega G\]
given respectively by projecting from $EG$ to $BG$ and by restricting from $D^2$ to $S^1$.
These are compatible with $h\colon\Omega G\xrightarrow{\simeq}\Omega^2 BG$ in the sense that $h\circ p_2\simeq p_1$.
We can now restate the second part of our theorem as follows.

\begin{theoremb}
The space $\varepsilon G$ is an $\calE$-algebra, and the maps $p_1$ and $p_2$ are weak equivalences of $\calE$-algebras.
\end{theoremb}

We recover a result of Menichi~\cite{\MenichiBVMorphism} which states that $h_\ast\colon H_\ast(\Omega G)\to H_\ast(\Omega^2 BG)$ is an isomorphism of BV-algebras.
(Menichi's result holds in the more general case that $G$ is a grouplike topological monoid.)

\skippy

\noindent 
The key ingredient in our paper is the following common feature of the operads $\fD$ and $\C$.
To \emph{elements} $a\in\fD(n)$ and $c\in\C(n)$ we can associate the \emph{spaces} $|a|$ and $|c|$ defined above, which we call the \emph{realizations} of $a$ and $c$.
The realizations come together with \emph{incoming boundary maps}
\[
\partial_1,\ldots,\partial_n\colon S^1\to |a|,\qquad \partial_1,\ldots,\partial_n\colon S^1\to |c|
\]
and \emph{outgoing boundary maps}
\[
\partial_\rmout\colon S^1\to |a|,\qquad \partial_\rmout\colon S^1\to |c|.
\]
Together, these realizations and boundary maps form what we call \emph{realization systems} for $\fD$ and $\C$.
This means that they are compatible with the symmetry maps, composition maps, and topology of the original operads $\fD$ and $\C$.
Most importantly, the realization $|x\circ_i y|$ of the composite of $x$ and $y$ is the pushout of $\partial_i\colon S^1\to |x|$ and $\partial_\rmout\colon S^1\to |y|$.
This fact is easily verified for both $\fD$ and $\C$, and leads directly to the definition of composition in $\calE$.

The paper contains a detailed development of realization systems.
The crucial result takes two operads $\calP$ and $\calQ$ with realization systems and constructs a \emph{mapping operad} $\calM$.
Points of $\calM(n)$ are triples $(p,q,f)$ where $p\in\calP(n)$, $q\in\calQ(n)$ and $f\colon |p|\to |q|$ is a map satisfying $f\circ\partial_i=\partial_i$ and $f\circ\partial_\rmout=\partial_\rmout$.
The operad $\calE$ is obtained in the case $\calP=\fD$, $\calQ=\C$ by restricting to the suboperad of $\calM$ consisting of triples $(p,q,f)$ in which $f$ is a homotopy equivalence.
The three algebras studied in the paper, namely the $\fD$-algebra $\Omega^2 BG$, the $\C$-algebra $\Omega G$, and the $\calE$-algebra $\varepsilon G$, are also described in terms of the realization systems for $\fD$ and $\C$.

\skippy 
\noindent The paper begins in Section~\ref{RealizationSection} with a discussion of realization systems, and proves that from a pair of operads with realization systems, one obtains a {mapping operad} $\calM$.
Sections~\ref{fDSection} and \ref{CSection} recall framed little discs and cacti, and describe their realization systems.
Then Section~\ref{ESection} applies the mapping operad construction to $\calP=\fD$ and $\calQ=\C$ to obtain $\calE$.

Section~\ref{AlgebrasSection} uses the realization systems to describe the actions of $\fD$ on $\Omega^2BG$, of $\C$ on $\Omega G$, and of $\calE$ on $\varepsilon G$.
With these ingredients in place Theorem~B follows without difficulty.

It remains to complete the proof of Theorem~A by showing that $\pi_1$ and $\pi_2$ are weak equivalences.
The structure of the proof is explained in detail in Section~\ref{ProofOfEquivalenceSection}.
It is shown there that Theorem~A follows from the computation of a certain long exact sequence, which is carried out in Section~\ref{LongExactSequenceSection}, and from a further result, Theorem~C.

Theorem~C states roughly that the projection $\pi_1\times\pi_2\colon\calE(n)\to\fD(n)\times\C(n)$ is a quasifibration.
It relies on some basic results on mapping spaces which are given in Section~\ref{MappingSpacesSection}, and on a study of the fibrewise structure of the realizations of $\fD$ and $\C$, which is carried out in Sections~\ref{RfDFibrewiseSection} and~\ref{RCFibrewiseSection}.
The proof of Theorem~C is completed in Section~\ref{QuasiFibrationProofSection}.

An appendix recalls some basic results of fibrewise topology.

\section{Realization systems for operads}\label{RealizationSection}

In this section we introduce \emph{realization systems} for an operad $\calO$.
A realization system with boundaries $X$ is a rule that assigns to each element $x\in\calO(n)$ a space $|x|$, the \emph{realization of $x$}, equipped with $(n+1)$ different maps $X\to |x|$.
The rule carries further structure and is subject to a number of axioms, all of which roughly amount to saying that the spaces $|x|$ are compatible with symmetry and composition in $\calO$, and vary continuously as $x$ varies within $\calO(n)$.
This is made precise in \S\ref{RealizationPropertiesSubsection}, and two simple examples are given in \S\ref{RealizationExamplesSubsection}.

Then in \S\ref{RealizationMappingSubsection} we define the \emph{mapping operad}.
Given operads $\calP$ and $\calQ$, both equipped with realizations, there is an associated mapping operad.
It lies over both $\calP$ and $\calQ$, and its elements are triples consisting of points $p\in\calP(n)$, $q\in\calQ(n)$, and a map $|p|\to|q|$ satisfying boundary conditions.
The operad $\calE$ of Theorem~A will be constructed as a suboperad of such a mapping operad.

For us, all operads are $\Sigma$-operads with a $0$-th space, which we do not insist be equal to a single point.
Composition within operads will be denoted using either the symbol $\gamma$ or the symbol $\circ_i$, for example $\gamma(x;y_1,\ldots,y_n)$ or $x\circ_i y$.

\subsection{Realization systems}\label{RealizationPropertiesSubsection}

Fix a compact, Hausdorff topological space $X$ and a topological operad $\calO$.
The following three definitions together define the notion of a realization system.
The definitions are somewhat abstract, and so the reader may find it helpful to bear in mind the example of framed little discs.
In this case, the realization $|a|$ of $a\in \fD(n)$ is the complement of the interiors of the little discs; the incoming boundary maps $S^1\to |a|$ are the boundaries of the little discs; and the outgoing boundary map $S^1\to |a|$ is the boundary of the big disc.

The first definition states what data must be given in order to define a realization system.
It is analogous to defining an operad by specifying spaces, symmetry maps, and composition maps.

\begin{definition}\label{RealizationsDefinition}
A \emph{realization system for $\calO$ with boundaries $X$} is a rule $\R\calO$ that for each $n\geqslant 0$ and each $x\in\calO(n)$ produces a topological space $|x|$ called the \emph{realization of $x$}, and that furthermore equips these spaces with the following structure:
\begin{itemize}
\item {\bf Boundary Maps.}
For each $x\in\calO(n)$, \emph{incoming boundary maps}
\[\partial_1,\ldots,\partial_n\colon X\to |x|\]
and an \emph{outgoing boundary map}
$\partial_\rmout\colon X\to |x|$.
\item
{\bf Symmetries.}
For each $x\in\calO(n)$ and $\sigma\in\Sigma_n$, a \emph{symmetry map}
$\sigma^\ast\colon |x|\to |x\sigma|$.
\item
{\bf Pasting and composition.}
Given $x\in\calO(n)$ and $y_i\in\calO(m_i)$ for $i=1,\ldots,n$, a pushout diagram
\begin{equation}\label{PastingDiagram}
\xymatrix@=35 pt{
\bigsqcup_{i=1}^n X\ar[r]^-{\bigsqcup\partial_\rmout}\ar[d]_{\bigsqcup\partial_i} & |y_1|\sqcup\cdots\sqcup|y_n|\ar[d] \\
|x|\ar[r] & |\gamma(x;y_1,\ldots,y_n)|
}
\end{equation}
called the \emph{pasting square}.
\end{itemize}
The system $\R\calO$ is required to satisfy the axioms of Definition~\ref{RealizationsAxiomsDefinition}, and to be topologized in the sense of Definition~\ref{RealizationsTopologyDefinition}, both below.
\end{definition}

Our second definition lists the axioms that the structures above must satisfy.
It is analogous to the axioms satisfied by the various composition and symmetry maps in an operad.
Although we list no fewer than six axioms, they are almost trivial to verify in the example of framed little discs.

\begin{definition}\label{RealizationsAxiomsDefinition}
A realization system $\R\calO$ must satisfy the axioms listed below.
In what follows we will use underlines to indicate tuples of elements.
For example, $\underline y$ will denote $y_1,\ldots,y_n$.
\begin{enumerate}
\item\label{UnitAxiom}
{\bf Unit:}
Taking $x=\mathbf{1}_\calO$, the boundary maps $\partial_1,\partial_\rmout\colon X\to |\mathbf{1}_\calO|$ are homeomorphisms and coincide.
Applying these homeomorphisms to \eqref{PastingDiagram} in the case $n=1$ and $x=\mathbf{1}_\calO$, and to \eqref{PastingDiagram} in the case $y_i=\mathbf{1}_\calO$, one obtains the diagrams
\[
\xymatrix{
X\ar[r]^{\partial_\rmout}\ar@{=}[d] & |y|\ar@{=}[d] \\
X\ar[r]_{\partial_\rmout} & |y|
}
\qquad\mathrm{and}\qquad
\xymatrix{
\bigsqcup X\ar@{=}[r]\ar[d]_{\bigsqcup\partial_i} & \bigsqcup X\ar[d]^{\bigsqcup\partial_i} \\
|x|\ar@{=}[r] & |x|
}
\]
respectively.

\item\label{SymmetriesAxiom}
{\bf Symmetries:}
The symmetry maps satisfy $\sigma^\ast\circ\tau^\ast=(\tau\sigma)^\ast$ and are compatible with the boundary maps in the sense that
\[\sigma^\ast\circ\partial_i = \partial_{\sigma^{-1}i},\quad\sigma^\ast\circ\partial_\rmout=\partial_\rmout.\]

\item\label{PastingBoundariesAxiom}
{\bf Pasting and boundaries:}
The lower and right-hand maps of diagram \eqref{PastingDiagram} are compatible with the boundary maps, in the sense that the following triangles commute.
\[
\xymatrix@=35 pt{
|x|\ar[r] &  |\gamma(x;\underline y)|  \\
X\ar[u]^{\partial_\rmout}\ar[ur]_{\partial_\rmout}  & {}
}
\qquad
\xymatrix@=35 pt{
\bigsqcup |y_i|\ar[d]      &      \bigsqcup\bigsqcup X\ar[l]_-{\bigsqcup\bigsqcup\partial_j} \ar[dl]^{\bigsqcup\bigsqcup\partial_{i,j}}  \\
|\gamma(x;\underline y)|  & {}
}
\]
The unions run over $i=1,\ldots,n$ and $j=1,\ldots,m_i$; the symbol $\partial_{i,j}$ denotes $\partial_{m_1+\cdots+m_{i-1}+j}$.

\item\label{PastingSymmetriesAxiomI}
{\bf Pasting and symmetries I:}
Given $\sigma_1\in\Sigma_{m_1},\ldots,\sigma_n\in\Sigma_{m_n}$, we can formulate diagram \eqref{PastingDiagram} for $x$ and $y_1,\ldots,y_n$, and for $x$ and $y_1\sigma_1,\ldots,y_n\sigma_n$.
Then the resulting squares are isomorphic under the relevant symmetry map, which means that the following diagram commutes.
\[\xymatrix@=15 pt{
{} &  \bigsqcup X\ar[dl]\ar[dr]\ar[dd]|{\hole}  &   {}  \\
\bigsqcup |y_i|\ar[rr]_(0.485){\sigma_1^\ast\sqcup\cdots\sqcup \sigma_n^\ast}\ar[dd] &   {}   &   \bigsqcup |y_i\sigma_i| \ar[dd]  \\
{}   &   |x|\ar[dl]\ar[dr]    &  {}  \\
|\gamma(x;\underline{y})|\ar[rr]_{(\sigma_1\oplus\cdots\oplus\sigma_n)^\ast}   &    {}   &   |\gamma(x;\underline{y\sigma})|
}\]
Here $\underline{ y\sigma}=y_1\sigma_1,\ldots,y_n\sigma_n$.

\item\label{PastingSymmetriesAxiomII}
{\bf Pasting and symmetries II:}
Similarly, given $\sigma\in\Sigma_n$ we can formulate diagram \eqref{PastingDiagram} for $x,y_1,\ldots,y_n$, and for $x\sigma,y_{\sigma^{-1}_1},\ldots,y_{\sigma^{-1}_n}$.
Then the two squares are isomorphic under the relevant symmetry maps, which means that the following diagram commutes.
\[\xymatrix@=15 pt{
\bigsqcup X\ar[rr]\ar[dd]\ar[dr]_{\sigma_\sqcup} & {} & \bigsqcup |y_{\sigma^{-1}_i}|\ar[dr]^{\sigma_\sqcup}\ar'[d][dd] & {} \\
{} & \bigsqcup X\ar[rr]\ar[dd] & {} & \bigsqcup |y_i| \ar[dd] \\
|x|\ar[dr]_{\sigma^\ast}\ar'[r][rr] & {} & |\gamma(x;\underline {y_{\sigma^{-1}}})|\ar[dr]_(0.4){\sigma(m_1,\ldots,m_n)^\ast\ \ \ \ } & {}\\
{} & |x\sigma|\ar[rr] & {} & |\gamma(x\sigma;\underline{y})|
}\]
Here $\underline{y_{\sigma^{-1}}}=y_{\sigma^{-1}_1},\ldots,y_{\sigma^{-1}_n}$ and $\sigma_\sqcup$ sends the $\sigma_i$-th cofactor to the $i$-{th} cofactor.

\item\label{AssociativityAxiom}
{\bf Associativity of pasting:}
Suppose we are given the data for an iterated composition in $\calO$.
Thus we have elements $x\in\calO(n)$, $y_i\in\calO(m_i)$ and $z^j_i\in\calO(l^j_i)$, for $i=1,\ldots,n$ and $j=1,\ldots,m_i$.
Then we can form composites
\[\gamma(x; \underline y),\qquad \gamma(y_i;\underline{z_i}),\qquad \gamma(\gamma(x;\underline y);\underline z),\qquad \gamma(x;\gamma(y_1;\underline{z_1}),\ldots,\gamma(y_n;\underline {z_n})).\]
The last two composites coincide, and consequently so do their realizations,
\[|\gamma(\gamma(x;\underline y);\underline z)|= |\gamma(x;\gamma(y_1;\underline{z_1}),\ldots,\gamma(y_n;\underline {z_n}))|.\]
By considering the pasting diagram \eqref{PastingDiagram} for each composite, we obtain \emph{two} maps from each of $|x|$, $|y_1|,\ldots,|y_n|$, $|z_1^1|,\ldots,|z_n^{m_n}|$ into the single space above.
The two maps in each such pair coincide.
\end{enumerate}
\end{definition}

The next definition says what it means to topologize a realization system, and makes use of \emph{fibrewise topology}.
The standard reference for this is \cite{\James}, and we have recalled the necessary ideas in Appendix~\ref{FibrewiseAppendix}.
A \emph{space fibred over $B$} is just a space $X$ equipped with a map $X\to B$, and the appendix recalls what it means for such a fibred space to be \emph{fibrewise compact} or \emph{fibrewise Hausdorff}.

\begin{definition}[Topology on a realization system]\label{RealizationsTopologyDefinition}
A \emph{topology} on a realization system $\R\calO$ consists of fibred spaces 
\[\rho_n\colon\R\calO(n)\to\calO(n)\]
for $n\geqslant 0$,
such that for each $x\in\calO(n)$ the fibre $\rho_n^{-1}(x)$ is identified with $|x|$.
Each $\R\calO(n)$ must be fibrewise compact Hausdorff over $\calO(n)$.
The boundary and symmetry maps determine functions
\[\partial_1,\ldots,\partial_n,\partial_\rmout\colon\calO(n)\times X\to\R\calO(n),\qquad \sigma^\ast\colon\R\calO(n)\to\sigma^\ast\R\calO(n)\]
between spaces over $\calO(n)$, and the pasting squares \eqref{PastingDiagram} become squares
\begin{equation}\label{FibrewisePastingDiagram}
\xymatrix@=35 pt{
\bigsqcup_{i=1}^n X \ar[r]^-{\bigsqcup\partial_\rmout}\ar[d]_{\bigsqcup\partial_i} & \R\calO(m_1)\sqcup\cdots\sqcup\R\calO(m_n)\ar[d]\\
\R\calO(n)\ar[r] & \R\calO(m_1+\cdots+m_n)
}\end{equation}
of spaces over $\calO(n)\times\calO(m_1)\times\cdots\times\calO(m_n)$.
All of the functions above are required to be continuous, and the squares \eqref{FibrewisePastingDiagram} are required to be fibrewise pushouts.
(The spaces written in \eqref{FibrewisePastingDiagram} should all be regarded as spaces over $\calO(n)\times\calO(m_1)\times\cdots\times\calO(m_n)$ by pulling back along the appropriate constant, projection or composition map.)
\end{definition}

\subsection{Examples of realizations}\label{RealizationExamplesSubsection}

This subsection attempts to illustrate the definition of realization systems by giving two simple examples in the case when $\calO$ is the operad of little $n$-cubes.
The paper's two main examples, which are the realizations for $\fD$ and $\C$, are presented in Sections \ref{fDSection} and \ref{CSection} respectively.

\begin{example}[The little cubes operad]
Let $C_n$ denote the little $n$-cubes operad.
Define a realization system $\R_1 C_n$ with boundaries $S^n$ as follows.
Given $x\in C_n(p)$, set $|x|=\bigvee_{i=1}^p S^n$.
\begin{itemize}
\item
The incoming boundary maps $\partial_1,\ldots,\partial_p\colon S^n\to\bigvee S^n$ are the standard insertions.
The outgoing boundary map $\partial_\rmout\colon S^n\to\bigvee S^n$ is the collapse map $\mathrm{coll}_x$ induced by $x$.
\item
The symmetry map $\sigma^\ast\colon |x|\to|x\sigma|$ is the map $\bigvee S^n\to\bigvee S^n$ that sends the $i$-th cofactor to the $(\sigma^{-1}_i)$-th cofactor.
\item
Given $x\in C_n(p)$ and $y_i\in C_n(q_i)$, the pasting square \eqref{PastingDiagram} is given by
\[\xymatrix@=35 pt{
\bigsqcup_{i=1}^p S^n \ar[r]^-{\bigsqcup\mathrm{coll}_{y_i}} \ar[d] & \bigsqcup_{i=1}^p\bigvee_{j=1}^{q_i} S^n \ar[d]\\
\bigvee_{i=1}^p S^n \ar[r]_-{\bigvee\mathrm{coll}_{y_i}} & \bigvee_{i=1}^p\bigvee_{j=1}^{q_i} S^n
}\]
where the vertical maps are the standard quotients.
\end{itemize}
The axioms of Definition~\ref{RealizationsAxiomsDefinition} hold trivially.
The system is topologized by setting $\R_1 C_n(p)=C_n(p)\times\bigvee S^n$, and the conditions of Definition~\ref{RealizationsTopologyDefinition} are immediately verified.
\end{example}

\begin{example}[Little cubes operads again]
Let $C_n$ again denote the little $n$-cubes operad.
Define the realization system $\R_2 C_n$ with boundaries $S^{n-1}$, which we regard as the boundary of the $n$-cube $I^n$, as follows.
Given $x\in C_n(p)$, define $|x|$ to be the complement in the big cube $I^n$ of the interiors of the little cubes of $x$.
\begin{itemize}
\item
The incoming boundary maps $\partial_1,\ldots,\partial_p\colon S^{n-1}\to |x|$ are the inclusions of the boundaries of the little cubes of $x$.
The outgoing boundary map $\partial_\rmout\colon S^{n-1}\to |x|$ is the inclusion of the boundary of the big cube.
\item
The symmetry map $\sigma^\ast\colon |x|\to |x\sigma|$ identifies the two spaces as subsets of the big cube.
\item
Given $x\in C_n(p)$ and $y_i\in C_n(q_i)$, the pasting square \eqref{PastingDiagram}
\[
\xymatrix@=35 pt{
\bigsqcup_{i=1}^n S^{n-1}\ar[r]^-{\bigsqcup\partial_\rmout}\ar[d]_{\bigsqcup\partial_i} & |y_1|\sqcup\cdots\sqcup|y_n|\ar[d] \\
|x|\ar[r] & |\gamma(x;y_1,\ldots,y_n)|
}
\]
is given as follows.
The bottom map is the inclusion of  $|x|$ into $|\gamma(x;y_1,\ldots,y_n)|$ given by regarding both as subsets of the big cube $I^n$.
The right-hand map sends $|y_i|$ into $|\gamma(x;y_1,\ldots,y_n)|$ by applying $x_i\colon I^n\to I^n$, the $i$-th little cube of $x$.
This is a pushout diagram.
\end{itemize}
The axioms of Definition~\ref{RealizationsAxiomsDefinition} are simple to  verify.
The system is topologized by defining $\R_2 C_n(p)$ to be the subspace of $C_n(p)\times I^n$ whose fibre over $x\in C_n(p)$ is precisely $|x|\subset I^n$.
The conditions of Definition~\ref{RealizationsTopologyDefinition} follow.
(Compare with the system $\R\fD$ in Section~\ref{fDRealizationsSubsection}.)
\end{example}

\subsection{The mapping operad}\label{RealizationMappingSubsection}

Suppose now that we are given operads $\calP$ and $\calQ$ equipped with realization systems $\R\calP$ and $\R\calQ$, both with boundaries $X$.
This subsection introduces the \emph{mapping operad} $\calM$ associated to this data.

\begin{definition}\label{MappingOperadDefinition}
Fix $n\geqslant 0$.
Then $\calM(n)$ denotes the set of triples $(a,c,f)$, where $a\in\calP(n)$, $c\in\calQ(n)$ and $f\colon |a|\to|c|$ is a continuous map satisfying the \emph{boundary conditions} $f\circ\partial_i=\partial_i$ and $f\circ\partial_\rmout=\partial_\rmout$.
\end{definition}

\begin{definition}
The action of $\Sigma_n$ on $\calM(n)$ is defined by the formula $(a,c,f)\sigma=(a\sigma,c\sigma,\sigma^\ast\circ f\circ {\sigma^{\ast}}^{-1})$,
where $\sigma^\ast\colon |a|\to|a\sigma|$ and $\sigma^\ast\colon|c|\to|c\sigma|$ are the symmetry maps.
\end{definition}

\begin{definition}\label{MappingOperadCompositionDefinition}
Composition in $\calM$ is defined by pasting functions in the following way.
Given $(a,c,f)\in\calM(n)$ and $(b_i,d_i,g_i)\in\calM(m_i)$ for $i=1,\ldots,n$, the composite
\[\gamma\left((a,c,f);(b_1,d_1,g_1),\ldots,(b_n,d_n,g_n)\right)\]
is the new triple
\[\left(\gamma(a;b_1,\ldots,b_n),\gamma(c;d_1,\ldots,d_n),\gamma(f;g_1,\ldots,g_n)\right),\]
where $\gamma(f;g_1,\ldots,g_n)$ is the unique map making the cube
\[\xymatrix@=15 pt{
\bigsqcup X\ar[rr]\ar[dd]\ar@{=}[dr] & {} & \bigsqcup |b_i|\ar[dr]^{\ \ g_1\sqcup\cdots\sqcup g_n}\ar'[d][dd] & {} \\
{} & \bigsqcup X\ar[rr]\ar[dd] & {} & \bigsqcup |d_i|\ar[dd] \\
|a|\ar[dr]_{f}\ar'[r][rr] & {} & |\gamma(a;\underline b)|\ar@{-->}[dr]_(0.4){\gamma(f;g_1,\ldots,g_n)\ \ } & {}\\
{} & |c|\ar[rr] & {} & |\gamma(c;\underline d)|
}\]
commute.
The front and rear faces of this cube are the pushout squares \eqref{PastingDiagram} for  $c,d_1,\ldots,d_n$ and $a,b_1,\ldots,b_n$ respectively, while  the left hand and upper faces come from the compatibility between pasting and boundary maps.
Thus a unique such $\gamma(f;g_1,\ldots,g_n)$ exists.
\end{definition}

To topologize $\calM(n)$ we use another notion from fibrewise topology.
Given spaces $U$ and $V$ fibred over $B$, we write $\Map_B(U,V)$ for the set of pairs $(b,f)$, where $b\in B$ is a point of the base and $f\colon U_b\to V_b$ is a map between the fibres over that point.
This set can be equipped with the \emph{fibrewise compact-open topology}, and we call the resulting space the \emph{fibrewise mapping space}.
See Appendix~\ref{FibrewiseAppendix} and the reference there.

The realization systems $\R\calP$ and $\R\calQ$ give us fibred spaces $\R\calP(n)\to\calP(n)$ and $\R\calQ(n)\to\calQ(n)$.
Write $\pi_1$ and $\pi_2$ for the projections of $\calP(n)\times\calQ(n)$ onto its factors.
Then the fibrewise mapping space $\Map_{\calP(n)\times\calQ(n)}(\pi_1^\ast\R\calP(n),\pi_2^\ast\R\calQ(n))$ consists of triples $(a,c,f)$ where $a\in\calP(n)$, $c\in\calQ(n)$ and $f\colon |a|\to |c|$.
In particular, it contains $\calM(n)$ as a subset.

\begin{definition}
The space $\calM(n)$ is topologized as a subspace of the fibrewise mapping space $\Map_{\calP(n)\times\calQ(n)}(\pi_1^\ast\R\calP(n),\pi_2^\ast\R\calQ(n))$.
\end{definition}

\begin{theorem}\label{MappingOperadIsAnOperadTheorem}
The collection $\calM=\{\calM(n)\}$ is topological operad, and the projections
\[\calP \longleftarrow \calM \longrightarrow \calQ\]
are morphisms of operads.
\end{theorem}

\begin{proposition}\label{SubOperadProposition}
Suppose that for both $\R\calP$ and $\R\calQ$ the pasting squares \eqref{PastingDiagram} are all \emph{homotopy} pushouts, rather than just pushouts.
Then the spaces
\[
\calM_\simeq (n)=\{(a,c,f)\in\calM(n)\mid f\ \mathrm{is\ a\ homotopy\ equivalence}\}
\]
form a suboperad $\calM_\simeq$ of $\calM$.
\end{proposition}

\begin{proof}[Proof of Proposition~\ref{SubOperadProposition}]
That $\calM_\simeq(1)$ contains the unit element, and that $\calM_\simeq$ is closed under the $\Sigma_n$ action, follow from the unit and symmetry axioms respectively.
Since homotopy equivalences are preserved by homotopy pushouts, it follows that if $f$ and $g_1,\ldots,g_n$ are all homotopy equivalences, then so is $\gamma(f;g_1,\ldots,g_n)$.
Thus $\calM_\simeq$ is closed under composition.
\end{proof}

\begin{proof}[Proof of Theorem~\ref{MappingOperadIsAnOperadTheorem}]
We begin by showing that $\calM$ is an operad in sets.
This is a direct consequence of the axioms for a realization system presented in Definition~\ref{RealizationsAxiomsDefinition}.
We will only show that composition is associative, the proof of the rest being similar but simpler.
Take elements $(x,u,f)\in\calM(n)$, $(y_i,v_i,g_i)\in\calM(m_i)$ and $(z_i^j,w_i^j,h_i^j)\in\calM(l_i^j)$ for $i=1,\ldots,n$ and $j=1,\ldots,m_i$.
Write
\[A=|\gamma(\gamma(x,\underline y);\underline z)|=|\gamma(x;\gamma(y_1;\underline{z_1}),\ldots,\gamma(y_n;\underline{z_n}))|\]
and
\[B=|\gamma(\gamma(u,\underline v);\underline w)|=|\gamma(u;\gamma(v_1;\underline{w_1}),\ldots,\gamma(v_n;\underline{w_n})|.\]
We must show that the two pasted maps
\begin{equation*}\label{PastedMapsEquation}
\gamma(\gamma(f,\underline g);\underline h)\quad \mathrm{and}\quad \gamma(f;\gamma(g_1;\underline{h_1}),\ldots,\gamma(g_n;\underline{h_n}))
\end{equation*}
coincide.
But both maps are characterized as the unique map making each square
\[\xymatrix{
|x|\ar[r]^f\ar[d]  & |u|\ar[d]\\
A\ar@{-->}[r] & B,
}
\qquad
\xymatrix{
|y_i|\ar[r]^{g_i}\ar[d]  & |v_i|\ar[d]\\
A\ar@{-->}[r] & B,
}
\qquad
\xymatrix{
|z_i^j|\ar[r]^{h^j_i}\ar[d]  & |w_i^j|\ar[d]\\
A\ar@{-->}[r] & B
}\]
commute.
Here the vertical maps come from Axiom~\ref{AssociativityAxiom} of Definition~\ref{RealizationsAxiomsDefinition}.
In particular, the two pasted maps coincide.

Now we will prove that the composition map $\calM(n)\times\calM(m_1)\times\cdots\times\calM(m_n)\to\calM(m_1+\cdots+m_n)$ is continuous.
The proof that the permutation map $\sigma\colon\calM(n)\to\calM(n)$ is continuous is similar.

We will use the following shorthand for spaces obtained from $\R\calP$, and the equivalent shorthand for spaces obtained from $\R\calQ$.
\begin{enumerate}
\item $\calP^\Pi$ denotes the product $\calP(n)\times\calP(m_1)\times\cdots\times\calP(m_n)$.
\item $\R\calP^\sqcup$ denotes the space obtained by pulling back each of $\R\calP(n)$ and $\R\calP(m_i)$ to $\calP^\Pi$ and then forming the disjoint union.
\item $\R\calP^\circ$ denotes the pullback of $\R\calP(m_1+\cdots+m_n)$ to $\calP^\Pi$.
\item $\xi_\calP\colon\R\calP^\sqcup\to\R\calP^\circ$ 
denotes the map obtained from the pasting squares~\eqref{FibrewisePastingDiagram}.
\end{enumerate}
Note that $\xi_\calP$ and $\xi_Q$ are proper fibrewise surjections, since $X$ is compact and the square \eqref{FibrewisePastingDiagram} is a fibrewise pushout.
(The analogous claim need not hold for $\xi_Q$.)

An element $x$ of $\calM(n)\times\calM(m_1)\times\cdots\times\calM(m_n)$ can be regarded as a point $x_1$ of the product
\[\Map_{\calP(n)\times\calQ(n)}(\R\calP(n),\R\calQ(n))\times\prod\Map_{\calP(m_i)\times\calQ(m_i)}(\R\calP(m_i),\R\calQ(m_i)),\]
which maps continuously into $\Map_{\calP^\Pi\times\calQ^\Pi}(\R\calP^\sqcup,\R\calQ^\sqcup)$, sending $x_1$ to a point $x_2$.
Now $\xi_\calP$ and $\xi_\calQ$ give continuous embeddings
\[\Map_{\calP^\Pi\times\calQ^\Pi}(\R\calP^\sqcup,\R\calQ^\sqcup)\xrightarrow{{\xi_\calQ}\ast}\Map_{\calP^\Pi\times\calQ^\Pi}(\R\calP^\sqcup,\R\calQ^\circ)\xleftarrow{\xi_\calP^\ast}\Map_{\calP^\Pi\times\calQ^\Pi}(\R\calP^\circ,\R\calQ^\circ).\]
The image of $x_2$ in the central space lifts to an element $x_3$ of the right hand space.
Finally we obtain a continuous map
\[\Map_{\calP^\Pi\times\calQ^\Pi}(\R\calP^\circ,\R\calQ^\circ)\to \Map_{\calP(\sum m_i)\times\calQ(\sum m_i)}(\R\calP({\textstyle\sum} m_i),\R\calQ(\textstyle\sum m_i))\]
sending $x_3$ to a point $y$.
This $y$ lies in $\calM(m_1+\cdots+ m_n)$.

The assignment $x\mapsto y$ just described is nothing more than the composition map, and the description makes its continuity evident.
This completes the proof.
\end{proof}

\section{Framed little discs and their realizations}\label{fDSection}

This section recalls Getzler's framed little discs operad $\fD$ and introduces its realization system $\R\fD$.

\subsection{Framed little discs}\label{FramedLittleDiscsRecap}

\begin{definition}[The framed little discs operad \cite{\Getzler}]
Getzler's \emph{framed little discs} operad $\fD$ has for its $n$-th space the collection $\fD(n)$ of embeddings
$a\colon\bigsqcup D^2\to D^2$
of $n$ \emph{little} discs into a single \emph{big} disc.
It is required that the $i$-th component $a_i\colon D^2\to D^2$ of the embedding be given by a combination of translation, rotation and rescaling.

The unit $\mathbf{1}_\fD$ is the identity embedding $D^2\to D^2$.
The action of $\Sigma_n$ on $\fD(n)$ is given by permuting the $n$ little discs: the $i$-th disc of $a\sigma$ is the $\sigma(i)$-th disc of $a$.
Composition in $\fD$ is given by composition of embeddings.
\end{definition}

\begin{example}\label{fDElementExample}
A typical element $a\in\fD(3)$ is drawn below.
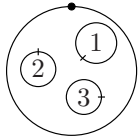
\begin{center}
\begin{lpic}[]{figures/DiscElement(,0.7in)}
\lbl{51.5,53.5;$1$}
\lbl{18.5,38.5;$2$}
\lbl{44.5,22;$3$}
\end{lpic}
\end{center}
The dash on the $i$-th little disc indicates the image of the basepoint of $S^1$, which we call the \emph{local marked point}.
The bullet $\bullet$ indicates the basepoint on the boundary of the big disc, which we call the \emph{global marked point}.
\end{example}

\subsection{Realizations of framed little discs}\label{fDRealizationsSubsection}

\begin{definition}[Realizations of framed little discs]
The realization system $\R\fD$ is defined as follows.
Given $a\in\fD(n)$, the \emph{realization} of $a$ is the space $|a|\subset D^2$
obtained by deleting the interiors of the little discs from the big disc.
\begin{itemize}
\item
The \emph{incoming boundary maps}
$\partial_1,\ldots,\partial_n\colon S^1\to |a|$
are given by $\partial_i=a_i|S^1$.
The \emph{outgoing boundary map}
$\partial_\rmout\colon S^1\to |a|$
is given by the inclusion $S^1\hookrightarrow D^2$.
In other words the $i$-th incoming boundary map is given by the boundary of the $i$-th little disc, while the outgoing boundary map is given by the boundary of the big disc.
\item
Given $a\in\fD(n)$ and $\sigma\in\Sigma_n$, the symmetry map
$\sigma^\ast\colon |a|\to |a\sigma|$
identifies $|a|$ and $|a\sigma|$ as subsets of $D^2$.
\item
Given $a\in\fD(n)$ and $b_i\in\fD(m_i)$ for $i=1,\ldots,n$, the pasting square
\begin{equation}\label{fDPastingDiagram}
\xymatrix{
\bigsqcup_{i=1}^n S^1\ar[r]^-{\bigsqcup\partial_\rmout} \ar[d]_{\bigsqcup\partial_i}   & |b_1|\sqcup\cdots\sqcup|b_n|\ar[d]\\
|a|\ar[r] & |\gamma(a;b_1,\ldots,b_n)|
}
\end{equation}
is given as follows.
The lower map is the inclusion of $|a|$ into $|\gamma(a;b_1,\ldots,b_n)|$ given by regarding both as subsets of $D^2$.
The right hand map sends $|b_i|$ into $|\gamma(a;b_1,\ldots,b_n)|$ by applying $a_i\colon D^2\to D^2$.
In other words it sends $|b_i|$ into the $i$-th little disc of $a$.
This is a pushout square, and in fact a homotopy pushout.
\end{itemize}
\end{definition}

\begin{example}
The realization of the element $a\in\fD(3)$ from Example~\ref{fDElementExample}, together with its boundary maps, is shown below.
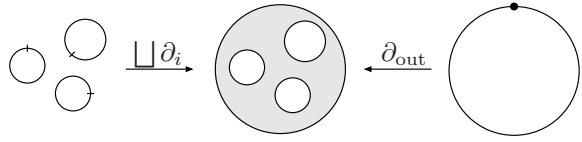
\begin{center}
\begin{lpic}{figures/DiscRealizationFigure(,0.7 in)}
\lbl{85,47;$\bigsqcup\partial_i$}
\lbl{225,47;$\partial_\rmout$}
\end{lpic}
\end{center}
Now take the following elements $a,b_1,b_2$.
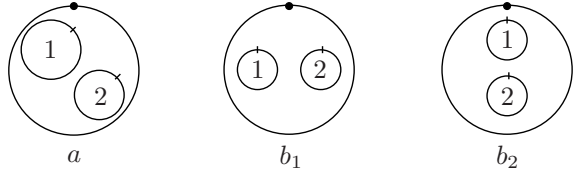
\begin{center}
\begin{lpic}[b(20 pt)]{figures/PastingDiscs(,0.7 in)}
\lbl{32,-10;$a$}
\lbl{136,-10;$b_1$}
\lbl{240,-10;$b_2$}
\lbl{20.5,42;$1$}
\lbl{44.5,19;$2$}
\lbl{120,32;$1$}
\lbl{150,32;$2$}
\lbl{240,46.5;$1$}
\lbl{240,19;$2$}
\end{lpic}
\end{center}
Then the pasting diagram~\eqref{fDPastingDiagram} for $\gamma(a;b_1,b_2)$ is depicted below:
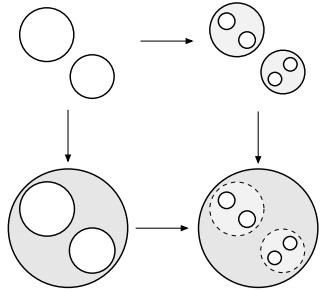
\begin{center}
\begin{lpic}{figures/PastingDiagram(,1.5 in)}
\end{lpic}
\end{center}
\end{example}

\begin{definition}[Topology on the realizations of framed little discs]
For $n\geqslant 0$ we define $\R\fD(n)$ to be the subspace of $\fD(n)\times D^2$ whose fibre over $a\in\fD(n)$ is precisely $|a|\subset D^2$.
It is equipped with the projection map $\rho_n\colon\R\fD(n)\to\fD(n)$.
As a closed subspace of $\fD(n)\times D^2$, it is a fibrewise compact Hausdorff space over $\fD(n)$.
\end{definition}

\begin{proposition}
The axioms of Definition~\ref{RealizationsAxiomsDefinition} and the conditions of Definition~\ref{RealizationsTopologyDefinition} are satisfied by the realizations of framed little discs.
\end{proposition}

\begin{proof}
The axioms of Definition~\ref{RealizationsAxiomsDefinition} follow from a lengthy but trivial verification.
The continuity required by Definition~\ref{RealizationsTopologyDefinition} follows immediately.
\end{proof}

\section{Cacti and their realizations}\label{CSection}

This section introduces the cacti operad and its realization system.
Cacti were introduced by Voronov in \cite{\Voronov}.
Several variants have been introduced since then, in particular by McClure-Smith~\cite{\McClureSmith}, Kaufmann~\cite{\Kaufmann}, and Salvatore~\cite{\Salvatore}.
We begin in \S\ref{VoronovCacti} with an informal introduction to cacti along the lines of Voronov's original definition.
Then in \S\ref{SalvatoreCacti} we recall in detail Salvatore's variant of cacti, which will be used in the rest of the paper.
Finally in \S\ref{CRealizations} we introduce the realization system for cacti.

\subsection{An introduction to cacti}\label{VoronovCacti}

This subsection gives an informal description of the cactus operad, along the lines of Voronov's original definition in \cite{\Voronov}.
We will not recall Voronov cacti in detail, but hope to give the unfamiliar reader an idea of the theory.

\begin{definition}
A \emph{cactus with $n$ lobes} is a planar treelike configuration of parameterized circles, of varying radii, labelled by $1,\ldots,n$, equipped with a \emph{global marked point} on one of the circles.
The individual circles are called the \emph{lobes} of the cactus, and the basepoints on the lobes are called the \emph{local marked points}.
See \cite{\Voronov} and \cite{\CohenVoronov}.

Thus a cactus consists of $n$ labelled circles of positive radius identified at finitely many \emph{intersection points}.
\emph{Planar} means that the circles meeting at each intersection point are equipped with a cyclic order.
\emph{Treelike} means that the space obtained by filling in each circle with a disc is contractible.

The  \emph{cactus operad} has for its $n$-th term the collection of all cacti with $n$ lobes.
\end{definition}

\begin{example}\label{CactusElementExample}
Here is a typical cactus with $4$ lobes.
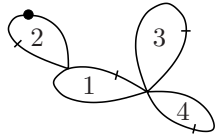
\begin{center}
\begin{lpic}[]{figures/CactusElement(,0.7 in)}
\lbl{35,21.5;$1$}
\lbl{13,42;$2$}
\lbl{65,42;$3$}
\lbl{75,10;$4$}
\end{lpic}
\end{center}
The bullet $\bullet$ indicates the global marked point, while the dashes indicate the local marked points.
\end{example}

Given cacti $c$ and $d$, the composed cactus $c\circ_i d$ is defined by the following pasting process.
Rescale $d$ so that its total length becomes the length of the $i$-th lobe of $c$.
Then $c\circ_i d$ is obtained from $c$ and $d$ by collapsing the $i$-th lobe of $c$ onto  $d$ using the \emph{pinch map} $S^1\to d$ that sends the basepoint to the global marked point and then proceeds around the `outside' of the cactus.

\subsection{Salvatore's cacti operad}\label{SalvatoreCacti}

In this subsection we recall the variant of the cacti operad developed by Salvatore in \cite{\Salvatore}.
The operad was denoted there by $f\mathrm{MS}$, but here we will call it $\C$.
It is a suboperad of the co-endomorphism operad $\CoEnd(S^1)$ of $S^1$.
We begin by establishing some notation.

\begin{definition}
\begin{enumerate}
\item
For $n\geqslant 1$ we define a space $\calF(n)$ as follows.
A point $x\in\calF(n)$ is a partition of $S^1$ into $n$ closed $1$-manifolds $I_j(x)$ with equal length and piecewise disjoint interiors, such that there is no cyclically ordered sequence $z_1,z_2,z_3,z_4$ with $z_1,z_3\in\interior I_j(X)$ and $z_2,z_4\in\interior I_k(x)$ with $j\neq k$.
Set $\calF(0)=\ast$.

\item
Given $x\in\calF(n)$ we define a map  $c(x)\colon S^1\to(S^1)^n$ as follows.
Given $j\in\{1,\ldots,n\}$, collapse each component of $\overline {S^1\setminus I_j(x)}$ to a point and rescale the quotient in order to identify it with $S^1$.
This gives a based map $\pi_j\colon S^1\to S^1$.
Then $c(x)$ is the product $\pi_1\times\cdots\times\pi_n$.

\item
Write $\Mon(I,\partial I)$ for the space of nondecreasing maps $I\to I$ that preserve $\partial I$.
Regard elements $f\in\Mon(I,\partial I)$ as based maps $f\colon S^1\to S^1$.

\item
Given $z\in S^1$ write $L_z\colon S^1\to S^1$ for translation by $z$.
\end{enumerate}
The constructions above induce an embedding
\[c\colon\calF(n)\times\Mon(I,\partial I)\times(S^1)^n\to\mathrm{CoEnd}(S^1)(n)\]
that sends $(x,f,\underline z)$ to the composite $(L_{z_1}\times\cdots\times L_{z_n})\circ c(x)\times f$.

\end{definition}

\begin{definition}[Salvatore's cacti operad]
$\C$ is defined to be the suboperad of $\CoEnd(S^1)$ whose $n$-th term is the image of the embedding $c\colon\calF(n)\times\Mon(I,\partial I)\times(S^1)^n\to\mathrm{CoEnd}(S^1)(n)$.
\end{definition}

\begin{note}
Here is the connection between the operad $\C$ and the informal notion of cacti given in the last subsection.
Fix $c\in\C(n)$ and write $c=c(x,f,\underline z)$.

Consider the image $c(S^1)\subset (S^1)^n$.
It depends only on $x$ and $\underline z$, and can be regarded as a Voronov cactus in which the lobes have equal length.
For $c(S^1)$ is a union of $n$ parameterized circles.
The $j$-th of these circles is simply $c(I_j(x))$.
This union of circles inherits a planar structure from the cyclic ordering of the components of the $I_j(x)$, and the condition on these cyclic orderings ensures that this union of circles is treelike.
The global marked point is given by $c(\ast)$.

Thus $c=c(x,f,\underline z)$ determines a cactus in which the lobes have equal length, depending only on $x$ and $\underline z$.
We regard $f$ as a reparametrization of the pinch map of this cactus; in the case when $f$ has constant speed $l_j$ on $I_j(x)$, we can regard $c$ as a cactus in which the $j$-th lobe has length $l_j$.
\end{note}

\begin{note}
Salvatore describes the precise relationship between $\C$ and Voronov's original definition of cacti.
In particular, Theorem~5.3.6 of \cite{\Kaufmann} holds with $\C$ in place of Voronov cacti.
\end{note}

\begin{note}\label{CellDecompositionNote}
Finally we recall the natural cell decomposition $\{ D_{\underline X} \}$ of $\calF(n)$.

Fix $x\in\calF(n)$.
Pull back the partition $\{I_j(x)\}$ of $S^1$ to a partition of $[0,1]$.
This partition is specified by boundary points $0=x_0<x_1<\cdots<x_k=1$ and a labelling $\underline X(x)=(X_1(x),\ldots X_k(x))$ of the intervals $[x_0,x_1],\ldots,[x_{k-1},x_k]$ by elements from $\{1,\ldots,n\}$.

The sequence $\underline X=\underline X(x)$ of the last paragraph satisfies the following three properties.
First, all values $1,\ldots,n$ appear.
Second, adjacent terms are distinct.
Third, there is no subsequence of the form $i,j,i,j$ with $i\neq j$.
Every such sequence $\underline X$ arises as $\underline X(x)$ for some element $x\in \calF(n)$.

Let $\underline X$ be a sequence satisfying the three conditions above.
Write $D_{\underline X}\subset \calF(n)$ for the subset consisting of those $x$ for which $\underline X(x)$ is contained in $\underline X$.
Then $D_{\underline X}$ is homeomorphic to the product $\prod_{j=1}^n\Delta^{d(j)}$, where $d(j)+1$ is the number of elements of $\underline X$ labelled by $j$.
The homeomorphism sends $x\in D_{\underline X}$ to the point whose image in $\Delta^{d(j)}$ records the lengths of the intervals labelled by $j$.
\end{note}

\subsection{The realizations of cacti}\label{CRealizations}

\begin{definition}[Realizations of cacti]
The realization system $\R\C$ is defined as follows.
Let $c\in\C(n)$ be a cactus with $n$ lobes.
The \emph{realization of $c$} is the topological space $c(S^1)$, which we denote by $|c|$.
\begin{itemize}

\item 
The \emph{$i$-th incoming boundary map} $\partial_i\colon S^1\to |c|$ is defined as follows.
Write $c=c(x,f,z)$.
Then $c(I_i(x))\subset |c|$ has the form $x_1^i\times\cdots\times S^1\times\cdots\times x_n^i$ for some \emph{lobe coordinate} $(x_1^i,\ldots,x_{i-1}^i,x_{i+1}^i,\ldots,x_n^i)$, and $\partial_i$ is $y\mapsto(x_1^i,\ldots,x_{i-1}^{i},y,x_{i+1}^i,\ldots,x_n^i)$.
The \emph{outgoing boundary map}  $\partial_\rmout\colon S^1\to |c|$ is given by $c$ itself.

\item
Let $c\in\C(n)$ and let $\sigma\in\Sigma_n$.
Then the homeomorphism $\sigma^\ast\colon |c|\to |c\sigma|$ is defined by $\sigma^\ast(t_1,\ldots,t_n)=(t_{\sigma 1},\ldots,t_{\sigma n})$.

\item
Let $c\in\C(n)$ and $d_i\in\C(m_i)$ for $i=1,\ldots,n$.
Then there is a pushout square
\begin{equation}\label{CactiPastingDiagram}
\xymatrix@=35 pt{
\bigsqcup_{i=1}^n S^1\ar[r]^-{\bigsqcup\partial_\rmout}\ar[d]_{\bigsqcup\partial_i} & |d_1|\sqcup\cdots\sqcup |d_n|\ar[d] \\
|c|\ar[r] & |\gamma(c;d_1,\ldots,d_n)|.
}
\end{equation}
The lower map sends $|c|=c(S^1)$ into $|\gamma(c;d_1,\ldots,d_n)|=(d_1\times\cdots\times d_n)c(S^1)$ via $d_1\times\cdots\times d_n$.
The right hand map sends $y\in|d_i|$ to \[(d_1(x^i_1),\ldots,d_{i-1}(x^i_{i-1}),y,d_{i+1}(x^i_{i+1}),\ldots,d_n(x^i_{n})).\]
\end{itemize}
\end{definition}

\begin{example}\label{CactusBoundaryExample}
The next figures illustrate the incoming
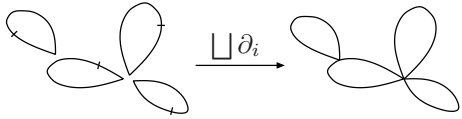
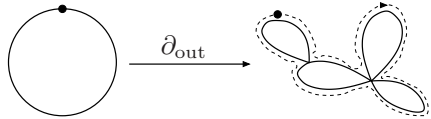
\begin{center}
\begin{lpic}[]{figures/CactusIncoming(,0.6 in)}
\lbl{120,37;$\bigsqcup\partial_i$}
\end{lpic}
\end{center}
and outgoing
\begin{center}
\begin{lpic}[]{figures/CactusOutgoing(,0.6 in)}
\lbl{95,38;$\partial_\rmout$}
\end{lpic}
\end{center}
boundary maps for the cactus of Example~\ref{CactusElementExample}.
\end{example}

\begin{definition}[Topology on the realizations of cacti]
We define $\R\C(n)$ to be the space of pairs
\[\R\C(n)=\{(c,\mathbf{z})\in\C(n)\times(S^1)^n\mid \mathbf{z}\in c(S^1)\}.\]
The fibre of the \emph{projection map} $\rho_n\colon\R\C(n)\to\C(n)$
over $c\in\C(n)$ is naturally identified with $|c|$.
As a closed subspace of $\C(n)\times(S^1)^n$, it is fibrewise compact and Hausdorff over $\C(n)$.
\end{definition}

It remains to show that the last two definitions do indeed endow $\C$ with a realization system $\R\C$.
This is given by the next proposition.

\begin{proposition}\label{CactiProposition}
\begin{enumerate}
\item 
Diagram \eqref{CactiPastingDiagram} is a pushout, and in fact a homotopy pushout.
\item
The realizations of cacti satisfy the conditions of Definition~\ref{RealizationsAxiomsDefinition}.
\item
The topology on realizations of cacti satisfy the conditions of Definition~\ref{RealizationsTopologyDefinition}.
\end{enumerate}
\end{proposition}

\begin{proof}
For the first part, a simple inductive argument shows that it is sufficient to prove the claim when all but one of the $d_i$ is the unit element.
To prove this case it suffices to show that, given $c\in\C(n)$ and $d\in\C(m)$, the diagram
\[\xymatrix{
S^1\ar[r]^{\partial_\rmout}\ar[d]_{\partial_i} & |d|\ar[d]^\beta\\
|c|\ar[r]_-\alpha & |c\circ_i d|
}\]
is a pushout.
This follows from the definitions and from the fact that $\partial_i$ is a cofibration.

We now prove the second part.
Axioms 1, 2, 3, 4 and 5 follow immediately from the definitions.
Let us turn to axiom 6.
Write $\Gamma=\gamma(\gamma(x,\underline y);\underline z)=\gamma(x;\underline{\gamma(y_i;\underline z_i)})$.
The compatibility of pasting and boundary maps means that the pairs of maps
\[u_1,u_2\colon |x|\to |\Gamma|,\quad
v_1,v_2\colon |y_i|\to |\Gamma|,\quad
w_1,w_2\colon |z_i^j|\to |\Gamma|\]
satisfy $u_p\circ\partial_\rmout=\partial_\rmout$, $v_p\circ\partial_\rmout=\partial_i$,
and $w_p\circ\partial_{k}=\partial_{L+k}$, where 
$L=l^1_1+\cdots+l_1^{m_1}+\cdots+l_i^1+\cdots+l_i^{j-1}$.
Since $\partial_\rmout\colon S^1\to |x|$, $\partial_\rmout\colon S^1\to |y_i|$ and $\bigsqcup \partial_k\colon\bigsqcup S^1\to |z_i^j|$ are all surjections, it follows that the two members of each pair of maps coincide.

Now we prove the final part.
Continuity of $\partial_i\colon\C(n)\times S^1\to\R\C(n)$ follows from the fact that the lobe coordinates of a cactus depend continuously on the cactus.
This can be proved using the cell decomposition of $\calF(n)$ given in Note~\ref{CellDecompositionNote}.
Continuity of $\partial_\rmout\colon\C(n)\times S^1\to\R\C(n)$ and $\sigma^\ast\colon\R\C(n)\to\R\C(n)$ is immediate.

Note that $\bigsqcup\partial_i\colon\C(n)\times\bigsqcup S^1\to\R\C(n)$ are continuous surjections between fibrewise compact Hausdorff spaces.
By Proposition~\ref{ContinuityProposition}, to prove continuity of the right-hand map of \eqref{FibrewisePastingDiagram} it will suffice to show that the composite with $\bigsqcup\partial_i$ is continuous.
But this follows from the compatibility of pasting with boundaries.
Similar reasoning shows that the lower map of \eqref{FibrewisePastingDiagram} is continuous.
\end{proof}

\section{The operad $\calE$}\label{ESection}

In Sections~\ref{fDSection} and~\ref{CSection} we saw that the framed little discs operad $\fD$ and the cacti operad $\C$ admit realization systems $\R\fD$ and $\R\C$ with boundaries $S^1$.
These systems both consist of fibrewise compact Hausdorff spaces.
We may therefore form the mapping operad $\calM$ of Section~\ref{RealizationSection} in the case $\calP=\fD$, $\calQ=\C$.
This is a topological operad, and since the pasting squares \eqref{PastingDiagram} for $\R\fD$ and $\R\C$ are homotopy pushouts, it has a suboperad $\calM_\simeq$ consisting of triples $(a,c,f)$ in which $f$ is a homotopy equivalence.

\begin{definition}
Set $\calE=\calM_\simeq$.
Thus $\calE(n)$ is the space of triples $(a,c,f)$, where $a\in\fD(n)$, $c\in\C(n)$, and $f\colon |a|\to|c|$ is a homotopy equivalence satisfying $f\circ\partial_i=\partial_i$ and $f\circ\partial_\rmout=\partial_\rmout$.
\end{definition}

An element $(a,c,f)\in\calE(2)$ is depicted in the introduction.
Theorem~\ref{MappingOperadIsAnOperadTheorem} immediately gives us the following result.
With this in hand, to prove Theorem~A it remains to show that the projection maps $\pi_1$ and $\pi_2$ are weak homotopy equivalences.

\begin{proposition}
$\calE$ is a topological operad, and the projections
\[\fD\xleftarrow{\quad \pi_1\quad}\calE\xrightarrow{\quad\pi_1\quad}\C\]
are morphisms of operads.
\end{proposition}

\section{Algebras from topological groups}\label{AlgebrasSection}

The aim of this paper is to compare the $\fD$-algebra $\Omega^2 BG$ with the $\C$-algebra $\Omega G$ by constructing an intermediate $\calE $-algebra $\varepsilon G$.
This section introduces all three algebras using the realization systems for $\fD$ and $\C$, and proves Theorem~B, which states that the three are weakly equivalent as $\calE $-algebras.

We begin in \S\ref{OmegaTwoBGSubsection} with $\Omega^2 BG$.
There is little to say in this case, but we include it for emphasis.
Then \S\ref{OmegaGSubsection} discusses $\Omega G$ and \S\ref{varepsilonGSubsection} discusses $\varepsilon G$.
Finally \S\ref{AlgebraProofsSubsection} gives some deferred proofs.

\subsection{The $\fD$-algebra $\Omega^2 BG$}\label{OmegaTwoBGSubsection}

In order to demonstrate how realization systems on an operad can be used to define algebras over that operad, we will 
briefly describe the $\fD$-algebra $\Omega^2 BG$.
Proofs are omitted.
First note that, for a fixed $a\in\fD(n)$, there is a pushout diagram
\[
\xymatrix{
\bigsqcup_{i=1}^n S^1\ar@{^{(}->}[r]\ar[d]_{\bigsqcup\partial_i} & \bigsqcup_{i=1}^n D^2\ar[d]^a \\
|a|\ar@{^{(}->}[r] & D^2.
}
\]

\begin{lemma}\label{nuPatchingLemma}
Given $a\in\fD(n)$ and $x_1,\ldots,x_n\in\Omega^2 BG$, there is a unique map $\xi\colon D^2\to BG$ with the following two properties.
\begin{enumerate}
\item
Its restriction to $|a|$ is constant.
\item
Its restriction to the $i$-th cofactor of $\bigsqcup_{i=1}^n D^2$ is given by $x_i$.
\end{enumerate}
\end{lemma}

\begin{definition}
Define $\nu_n\colon\fD(n)\times(\Omega^2 BG)^n\to\Omega^2 BG$ by
\[\nu_n(a;x_1,\ldots,x_n)=\xi,\]
where $\xi$ is the map of Lemma~\ref{nuPatchingLemma}.
\end{definition}

\begin{proposition}
The collection $\nu=\{\nu_n\}$ makes $\Omega^2 BG$ into a $\fD$-algebra.
\end{proposition}

\subsection{The $\C$-algebra $\Omega G$}\label{OmegaGSubsection}

We now explain how to make $\Omega G$ into a $\C$-algebra.
This algebra structure was first obtained by Salvatore in \cite{\Salvatore} using his proof of the topological cyclic Deligne conjecture.
We will describe it directly in terms of the realizations $\R\C$.

\begin{notation}
Let $X$ be a space and let $Y$ be a $G$-space.
Given a continuous map $f\colon X\to Y$ and an element $g\in G$, we write $g\cdot f\colon X\to Y$ for the map that sends $x\in X$ to $g\cdot f(x)$.
Such maps will be called \emph{left translates} of $f$.
\end{notation}

\begin{lemma}\label{PatchingLemma}
Let $c$ be a cactus with $n$ lobes and let $\gamma_1,\ldots,\gamma_n$ be elements of $\Omega G$.
Then there is a unique map $\alpha\colon |c|\to G$ satisfying:
\begin{enumerate}
\item
$\alpha(\bullet)=e$, where $\bullet$ denotes the global marked point in $|c|$;
\item
for each $i=1,\ldots,n$, the composite $\alpha\circ\partial_i\colon S^1\to G$ is a left translate of $\gamma_i$.
\end{enumerate}
\end{lemma}
\begin{proof}
The realization $|c|$ is a treelike configuration of circles, so after relabelling the lobes we may assume that $\partial_1(S^1)$ contains $\bullet$ and that each $\partial_{i+1}(S^1)$ meets $\partial_1(S^1)\cup\cdots\cup\partial_i(S^1)$ in a single point.
There is a unique $\alpha^1\colon\partial_1(S^1)\to G$ sending $\bullet$ to $e$ and such that $\alpha^1\circ\partial_1$ is a left translate of $\gamma_i$.
It has a unique extension $\alpha^2\colon\partial_1(S^1)\cup\partial_2(S^1)\to G$ for which $\alpha^2\circ\partial_2$ is a left translate of $\gamma_2$.
Proceeding in this way, the claim follows.
\end{proof}

\begin{definition}
Define
$\omega_n\colon\C(n)\times(\Omega G)^n\to\Omega G$
by 
\[\omega_n(c;\gamma_1,\ldots,\gamma_n)=\alpha\circ\partial_\rmout,\]
where $\alpha\colon |c|\to G$ is the map of Lemma~\ref{PatchingLemma}.
\end{definition}

\begin{proposition}\label{OmegaGIsAnAlgebraProposition}
The collection $\omega=\{\omega_n\}$ makes $\Omega G$ into a $\C$-algebra.
\end{proposition}

The proof will be given in \S\ref{AlgebraProofsSubsection}.
The action $\omega$ encodes some natural properties of $\Omega G$.
For example, given $s\in S^1$, let $c_s\in\C(1)$ denote the cactus in which $\bullet=\partial_1(s)$ and $\partial_\rmout$ has constant speed.
Then
\[\omega_1(c_-,-)\colon S^1\times\Omega G\to\Omega G\]
is the circle action that sends $(s,\gamma)$ to the loop
$t\mapsto \gamma(s)^{-1}\cdot\gamma(s+t)$.
This circle action was considered by Menichi in \cite{\MenichiBVMorphism} and by Salvatore in \cite{\Salvatore}. 
Now let $c_P\in\C(2)$ denote the cactus in which two circles of radius $1/2$ are joined at their basepoint, and the global marked point lies on the first circle at its basepoint:
Then the map
\[\omega_2(c_P;-,-)\colon\Omega G\times\Omega G\to\Omega G\]
is the ordinary Pontrjagin product given by concatenating loops.
More generally, there is a natural inclusion of the little intervals operad into $\C$, under which the algebra structure $\omega$ pulls back to the standard $E_1$-algebra structure on $\Omega G$.

\subsection{The $\calE$-algebra $\varepsilon G$}\label{varepsilonGSubsection}

Now we will define the space $\varepsilon G$.
We will give it the structure of an $\calE$-algebra, and show that it is weakly equivalent to both $\Omega^2 BG$ and $\Omega G$ as $\calE$-algebras.
The algebra structure on $\varepsilon G$ is an almost literal mixture of the existing algebra structures on $\Omega^2 BG$ and $\Omega G$.

Let $EG\to BG$ be the universal principal $G$ bundle.
Fix a basepoint $\ast$ of $EG$ and use it to define a basepoint of $BG$ and an inclusion $\iota_G\colon G\hookrightarrow EG$, $g\mapsto g \cdot\ast$.

\begin{definition}\label{omegaGDefinition}
Let $\varepsilon G$ be the space of maps $\phi\colon D^2\to EG$ for which $\phi|_{S^1}$ factors through $\iota_G$, and for which $\phi(\ast)=\iota_G(e)$.
There are projections
\[p_1\colon \varepsilon G\to\Omega^2 BG,\qquad p_2\colon \varepsilon G\to\Omega G,\]
where $p_1$ is obtained by projecting from $EG$ to $BG$, and $p_2$ is defined by $\iota_G\circ p_2(\phi)=\phi|_{S^1}$.
\end{definition}

We will again make use of the pushout diagram
\begin{equation}\label{RealizationDiscGluingDiagram}
\xymatrix{
\bigsqcup_{i=1}^n S^1\ar@{^{(}->}[r]\ar[d]_{\bigsqcup\partial_i} & \bigsqcup_{i=1}^n D^2\ar[d]^a \\
|a|\ar@{^{(}->}[r] & D^2
}
\end{equation}
determined by an element $a\in\fD(n)$.

\begin{lemma}\label{epsilonPatchingLemma}
Given $(a,c,f)\in\calE(n)$ and $\phi_1,\ldots,\phi_n\in\varepsilon G$,
there is a unique map $\phi\colon D^2\to EG$ with the following properties:
\begin{enumerate}
\item
Its restriction to $|a|$ factors through $f\colon |a|\to|c|$.
\item
Its restriction to the $i$-th cofactor of $\bigsqcup_{i=1}^n D^2$ is a left translate of $\phi_i$.
\item
Its value on $\ast\in D^2$ is $\iota_G(e)$.
\end{enumerate}
Note that $\psi\in\varepsilon G$.
\end{lemma}

\begin{proof}
The conditions mean that $\psi|_{|a|}$ must factor as $\iota_G\circ\alpha\circ f$,
where $\alpha$ is the map obtained by applying Lemma~\ref{PatchingLemma} to $c$ and $p_2(\phi_1),\ldots,p_2(\phi_n)$.
There is a unique extension of this map to $D^2$ satisfying the second condition.
\end{proof}

\begin{definition}\label{omegaGAlgebraDefinition}
Define $\varepsilon_n\colon\calE(n)\times(\varepsilon G)^n\to\varepsilon G$ by
\[\varepsilon_n\left((a,c,f),(\phi_1,\ldots,\phi_n)\right)=\psi,\]
where $\psi$ is the map of Lemma~\ref{epsilonPatchingLemma}.
\end{definition}

The following two propositions together prove Theorem~B.
Their proofs are given in the next subsection.

\begin{proposition}\label{varepsilonGIsAnAlgebraProposition}
The collection $\varepsilon=\{\varepsilon_n\}$ makes $\varepsilon G$ into a $\calE$-algebra.
\end{proposition}

\begin{proposition}\label{varepsilonGIsEquivalentProposition}
The projection maps $p_1\colon\varepsilon G\to\Omega^2 BG$, $p_2\colon \varepsilon G\to\Omega G$ are homotopy equivalences of $\calE$ algebras.
\end{proposition}

\subsection{Proof of Propositions \ref{OmegaGIsAnAlgebraProposition}, \ref{varepsilonGIsAnAlgebraProposition} and \ref{varepsilonGIsEquivalentProposition}}\label{AlgebraProofsSubsection}

In this subsection we resume the practice of using underlines to indicate tuples of elements.
For example, given $\gamma^i_1, \ldots,\gamma^i_{m_i}$ for each $i=1,\ldots,n$, we write $\underline\gamma$ for $\gamma^1_1,\ldots,\gamma^n_{m_n}$ and we write $\underline\gamma^i$ for $\gamma^i_1,\ldots,\gamma^i_{m_i}$.

\begin{proof}[Proof of Proposition~\ref{OmegaGIsAnAlgebraProposition}]
Axioms 1, 2 and 3 of Definition~\ref{RealizationsAxiomsDefinition} show that the unit element $\mathbf{1}_\C$ acts as the identity on $\Omega G$, and that 
the equivariance property $\omega_n(c\sigma;\gamma_1,\ldots,\gamma_n)=\omega_n(c;\gamma_{\sigma^{-1}1},\ldots,\gamma_{\sigma^{-1}(n)})$ holds.

Now let us show that $\omega$ is associative.
Let $c\in\C(n)$, $d_i\in\C(m_i)$ and $\gamma^i_1,\ldots,\gamma^i_{m_i}\in\Omega G$ for $i=1,\ldots,n$.
We must show that
\[
\omega_{m_1+\cdots+m_n}(\gamma(c;\underline d);\underline \gamma)
=
\omega_n(c;\underline{\omega_{m_i}(d_i;\underline{\gamma^i})}).
\]
We can use Lemma~\ref{PatchingLemma} to obtain three different maps.
\begin{enumerate}
\item
Let
$\alpha\colon |\gamma(c;d_1,\ldots,d_n)|\to G$
denote the map obtained using
$\gamma(c;d_1,\ldots,d_n)$ and $\gamma^1_1,\ldots,\gamma^n_{m_n}$.
\item
Let $\beta\colon |c|\to G$
denote the map obtained using
$c$ and the $\omega_{m_i}(d_i;\gamma^i_1,\ldots,\gamma^i_{m_i})$.
\item
Let $\gamma^i\colon |d_i|\to G$ denote the map obtained using 
$d_i$ and $\gamma^i_1,\ldots,\gamma^i_{m_i}$.
\end{enumerate}
We must show that $\alpha\circ\partial_\rmout=\beta\circ\partial_\rmout$.
We can find $g_1,\ldots,g_n\in G$ and two commutative squares:
\[\xymatrix@=35 pt{
\bigsqcup S^1\ar[r]^-{\bigsqcup\partial_\rmout}\ar[d]_{\bigsqcup\partial_i} & |d_1|\sqcup\cdots\sqcup|d_n|\ar[d] ^{\bigsqcup g_i\cdot\gamma^i} \\
|c|\ar[r]_-\beta & G
}
\qquad\qquad
\xymatrix@=35 pt{
\bigsqcup S^1\ar[r]^-{\bigsqcup\partial_\rmout}\ar[d]_{\bigsqcup\partial_i} & |d_1|\sqcup\cdots\sqcup|d_n|\ar[d] ^\mu \\
|c|\ar[r]_-\lambda & |\gamma(c;d_1,\ldots,d_n)|
}\]
The first square comes from the definition of $\beta$.
The second is a pasting square \eqref{PastingDiagram}, and in particular is a pushout.
Comparing the squares we see that $\beta$ factors as $\alpha'\circ\lambda$ for some $\alpha'\colon |\gamma(c;d_1,\ldots,d_n)|\to G$.
The compatibility of $\lambda$ and $\mu$ with the boundary maps means that $\alpha'$ in fact satisfies the properties that characterize $\alpha$, so that $\alpha'=\alpha$.
Thus $\beta=\alpha\circ\lambda$, and by the compatibility of $\lambda$ with the boundary maps we have $\beta\circ\partial_\rmout=\alpha\circ\lambda\circ\partial_\rmout=\alpha\circ\partial_\rmout$ as required.

Let us show that $\omega_n$ is continuous.
Lemma~\ref{PatchingLemma} gives us a function
\[\alpha\colon\R\C(n)\times(\Omega G)^n\to G\]
and it suffices to show that this is continuous.
The map $\bigsqcup\partial_i\colon \bigsqcup S^1\times\C(n)\to \R\C(n)$ is proper surjection of fibred spaces, so by Proposition~\ref{ContinuityProposition} it will suffice to show that each composite $\alpha\circ\partial_i$
is continuous.
This composite is given by $(t,c,\gamma_1,\ldots,\gamma_n)\mapsto g_i\cdot\gamma_i(t)$ for some $g_i\in G$.
It will therefore suffice to show that the assignments $\C(n)\times(\Omega G)^n\to G$, $(c,\gamma_1,\ldots,\gamma_n)\mapsto g_i$ are continuous.
These maps factor through a map $\calF(n)\times (\Omega G)^n\to G$, whose continuity can be proved using the cell decomposition of $\calF(n)$ given in Note~\ref{CellDecompositionNote}.
\end{proof}

\begin{proof}[Proof of Proposition~\ref{varepsilonGIsAnAlgebraProposition}]
To begin we must show that $\varepsilon$ satisfies the unit, equivariance and associativity rules.
We will only prove associativity as the other two properties can be proved in a similar but much simpler way.
Take $(a,c,f)\in\calE(n)$ and $(b_i,d_i,g_i)\in\calE(m_i)$ for $i=1,\ldots,n$, and write their composite as $(A,C,F)$.
Take $\phi_j^i\in\varepsilon G$ for $i=1,\ldots,n$ and $j=1,\ldots,m_i$.
We must show that
\begin{equation} \label{AssociativityRequirementEquation}
\varepsilon_{m_1+\cdots+m_n}\left((A,C,F);\underline\phi \right)
=
\varepsilon_n\left((a,c,f); \underline {\varepsilon_{m_i}\left((b_i,d_i,g_i);\underline{\phi^i}\right)}\right).
\end{equation}
Consider the following diagram of pushout squares.
\[\xymatrix{
{}  & \bigsqcup_i\bigsqcup_j S^1\ar[d]\ar[r] & \bigsqcup_i\bigsqcup_j D^2 \ar[d] \\
\bigsqcup_i S^1\ar[r]\ar[d] & \bigsqcup_i |b_i|\ar[r]\ar[d] &   \bigsqcup_i D^2 \ar[d]  \\
|a|\ar[r] & |A|\ar[r] & D^2
}\]
The upper square comes from \eqref{RealizationDiscGluingDiagram} for the $b_i$;
the right hand two squares compose to give \eqref{RealizationDiscGluingDiagram} for $A$, and the lower two squares compose to give \eqref{RealizationDiscGluingDiagram} for $a$; the left hand square is the pasting square for $A$.
Using this diagram we can characterise the right hand side of \eqref{AssociativityRequirementEquation} as the map $\psi\colon D^2\to EG$ that
\begin{enumerate}
\item
when restricted to  $|a|$, factors through $f\colon |a|\to |c|$;
\item
when restricted to $|b_i|$,  factors through $g_i\colon|b_i|\to|d_i|$;
\item
on the $(i,j)$-th cofactor $D^2$ is a left translate of $\phi_j^i$;
\item
sends $\ast$ to $\iota_G(e)$.
\end{enumerate}
The first two of these properties, together with the definition of $F$, show that the restriction of $\psi$ to $|A|$ factors through $F\colon |A|\to|C|$.
This property, together with with the third and fourth properties above, characterize the left hand side of \eqref{AssociativityRequirementEquation}.
Equality follows.
Thus $\varepsilon$ makes $\varepsilon G$ into a $\calE$-algebra in sets.

To show that $\varepsilon$ makes $\varepsilon G$ into a $\calE$-algebra in the topological setting
we will show that $\varepsilon_n\colon\calE(n)\times(\varepsilon G)^n\to\varepsilon G$ is continuous.
To do this we will show that the adjoint $D^2\times\calE(n)\times(\varepsilon G)^n\to EG$ is continuous.
Since the lower and right-hand maps of diagram  \eqref{RealizationDiscGluingDiagram} give inclusions of closed subsets
$\bigsqcup D^2\times\calE(n)\hookrightarrow D^2\times\calE(n)$ and $\pi_1^\ast\R\fD(n)\hookrightarrow D^2\times\calE(n)$, 
it suffices to show continuity of the restrictions
\[\xi\colon\bigsqcup D^2\times\calE(n)\times(\varepsilon G)^n\to EG,
\qquad
\eta\colon\pi_1^\ast\R\fD(n)\times(\varepsilon G)^n\to EG.\]
For $\xi$ this follows from continuity of the map $((a,c,f),(\phi_1,\ldots,\phi_n))\mapsto (g_1,\ldots,g_n)$ in the proof of Proposition~\ref{OmegaGIsAnAlgebraProposition}.
For $\eta$ it follows from continuity of two maps
\[\alpha\colon\R\C(n)\times(\Omega G)^n\to G,
\qquad
\beta\colon\pi_1^\ast\R\fD(n)\to\R\C(n).\]
Here $\beta$ is given in the fibre over $(a,c,f)$ by $x\mapsto f(x)$,
and is continuous by part~5 of Proposition~\ref{FibrewiseMappingProposition},
while $\alpha$ was defined, and shown to be continuous, in the proof of Proposition~\ref{OmegaGIsAnAlgebraProposition}.
\end{proof}

\begin{proof}[Proof of Proposition~\ref{varepsilonGIsEquivalentProposition}]
It is immediate that both $p_1$ and $p_2$ are $\calE$-algebra morphisms.
We must now show that both $p_1$ and $p_2$ are homotopy equivalences.

We first show that $p_1$, $p_2$ are fibrations.
For $p_2$ the homotopy lifting problem is adjoint to an extension problem that can always be solved.
For $p_1$ the homotopy lifting problem is adjoint to the problem of finding a lift $F$ in a diagram of the form
\[
\xymatrix{
X\times D^2\times\{0\}\ar[r]^-{\widetilde h}\ar[d] &  EG\ar[d]\\
X\times D^2\times I\ar[r]_-{\widetilde H}\ar@{-->}[ur]_F&  BG
}\]
with the additional property that $F|(X\times \ast\times I)$ is constant with value $\iota_G(e)$.
To solve this, first choose an arbitrary lift $F'$.
Write $F'|X\times \ast\times I=\iota_G\circ f$, and then define $F$ by $F(x,d,s)=F'(x,d,s)\cdot f(x,s)^{-1}$.

It now suffices to show that $\pi_1$ and $\pi_2$ have contractible fibres.
But $p_1^{-1}(x)$ is homeomorphic to the space of based maps $D^2\to G$, while $p_2^{-1}(\gamma)$ is the space of maps $D^2\to EG$ whose restriction to $S^1$ is $\iota_G\circ\gamma$.
These are contractible, by contractibility of $D^2$ and $EG$ respectively.
\end{proof}

\section{Proof of Theorem~A}\label{ProofOfEquivalenceSection}

This section gives the proof of Theorem~A, stating certain results that will be proved in the remaining sections of the paper.  The theorem states that the projections $\pi_1\colon\calE\to\fD$ and $\pi_2\colon\calE\to\C$ are weak homotopy equivalences of operads.
This means that for each $n\geqslant 0$, the maps $\pi_1\colon\calE(n)\to\fD(n)$ and $\pi_2\colon\calE(n)\to\C(n)$ are weak equivalences of spaces.
To show this we will consider the product
\[\Pi\colon \calE(n)\to\fD(n)\times\C(n)\]
of $\pi_1$ and $\pi_2$.
It would be ideal for us if $\Pi$ were a fibration, but this is not the case: a typical fibre is nonempty, but the fibre over a pair $(a,c)$ can be empty if the little discs of $a$ meet the boundary of the big disc, for then the boundary conditions can be overdetermined.
This issue is easily remedied.

\begin{definition}
Let $\fD^\circ(n)\subset\fD(n)$ consist of those elements whose little discs do not meet the boundary of the big disc, and let $\calE^\circ(n) = \pi_1^{-1}\fD^\circ(n)\subset\calE(n)$.
\end{definition}

\begin{lemma}
The inclusions $\fD^\circ(n)\hookrightarrow\fD(n)$ and $\calE^\circ(n)\hookrightarrow\calE(n)$ are homotopy equivalences.

\end{lemma}

\begin{proof}
Given $a\in\fD^\circ(n)$, write $a_{1/2}$ for the element obtained by halving the radii of the little discs of $a$.
Write $\rho_a\colon |a_{1/2}|\to|a|$ for the map that preserves $|a|\subset|a_{1/2}|$ and that projects points inside a little disc of $a$ to the boundary of that little disc.
Then the maps $\fD(n)\to\fD^\circ(n)$, $a\mapsto a_{1/2}$ and $\calE(n)\to\calE^\circ(n)$, $(a,c,f)\mapsto(a_{1/2},c,f\circ\rho_a)$ are homotopy inverse to the inclusions above.
\end{proof}

To prove Theorem~A it will therefore suffice to prove that the restrictions
$\pi_1\colon\calE^\circ(n)\to\fD^\circ(n)$ and $\pi_2\colon\calE^\circ(n)\to\C(n)$ are weak equivalences.
To do this we now consider the map
\[\Pi\colon \calE^\circ(n)\to\fD^\circ(n)\times\C(n),\]
which has much better fibrewise properties than its predecessor.

\begin{theoremc}
Fix $(a,c)\in\fD^\circ(n)\times\C(n)$ and write  $\Map_\partial(|a|,|c|)$ for the fibre of $\calE^\circ(n)$ over $(a,c)$.
Then there is a neighbourhood $U$ of $(a,c)$ over which we can find a fibrewise homotopy equivalence
\[\calE^\circ(n)|U\simeq U\times\Map_\partial(|a|,|c|).\]
\end{theoremc}

Thus each point of $\fD^\circ(n)\times\C(n)$ has arbitrarily small neighbourhoods over which $\Pi$ is a quasifibration.
By Corollary 2.4 of \cite{\MayQuasi}, such maps are themselves quasifibrations.

\begin{corollary}
The projection $\Pi\colon\calE^\circ(n)\to\fD^\circ(n)\times\C(n)$ is a quasifibration.
\end{corollary}

We will use the quasifibration $\Pi$ to show that $\pi_1$ and $\pi_2$ are weak equivalences.
Write $PRB_n$ for the pure ribbon braid group.
Then both $\fD^\circ(n)$ and $\C(n)$ are Eilenberg-MacLane spaces $K(PRB_n;1)$.
(The relevant background is recalled in \S\ref{PureRibbonBraidGroupSubsection} and \S\ref{DiscsCactiPureRibbonBraidGroupSubsection}.)
It will therefore suffice to show that $\pi_i(\calE^\circ(n))$ is trivial for $i\neq 1$, and that $\Pi_\ast\colon\pi_1(\calE^\circ(n))\to PRB_n\times PRB_n$ identifies $\pi_1(\calE^\circ(n))$ with the diagonal subgroup of $PRB_n\times PRB_n$.

Write $\mathcal{F}_n$ for a typical fibre of $\Pi$.
For $\Pi$ to be a quasifibration means that $\mathcal{F}_n$ is weakly equivalent to the homotopy fibre of $\Pi$, or in other words that there is a long exact sequence of homotopy groups:
\begin{equation}\label{LongExactSequenceI}\cdots\longrightarrow\pi_i(\mathcal{F}_n)\longrightarrow\pi_i(\calE^\circ(n))\longrightarrow\pi_i(\fD^\circ(n))\times\pi_i(\C(n))\longrightarrow\cdots\end{equation}
We will establish the criteria listed in the last paragraph by exploiting this long exact sequence.

\begin{proposition}\label{FibreHomotopyTypeProposition}
There is a homotopy equivalence $\mathcal{F}_n\to PRB_n$.
In particular $\mathcal{F}_n$ has contractible components.
\end{proposition}

It follows immediately from \eqref{LongExactSequenceI} that $\pi_i(\calE^\circ(n))$ is trivial for $i>1$, and that $\pi_1(\calE^\circ(n))$ and $\pi_0(\calE^\circ(n))$ are computed using an exact sequence
\begin{equation}\label{LongExactSequenceII}
0\to\pi_1(\calE^\circ(n))\xrightarrow{p} PRB_n\times PRB_n\xrightarrow{q} PRB_n\xrightarrow{r}\pi_0(\calE^\circ(n))\to 0.
\end{equation}
There is an action of $PRB_n\times PRB_n$ on $PRB_n$, and $q$ is given by applying this action to the constant element $\ast$. 
Exactness at the last two terms of the sequence means that $q^{-1}(\ast)$ is the image of $p$ and $r^{-1}(\ast)$ is the image of $q$.
We can therefore show that $\pi_0(\calE^\circ(n))=\ast$ and that $p$ identifies $\pi_1(\calE^\circ(n))$ with the diagonal subgroup of $PRB_n\times PRB_n$ by proving the next proposition, which completes the proof of Theorem~A.

\begin{proposition}\label{ActionOnTheFibreProposition}
The action of $PRB_n\times PRB_n$ on $PRB_n$ is given by $(\gamma,\delta)\cdot\phi=\delta\cdot\phi\cdot\gamma^{-1}$.
\end{proposition}

The promised recollections on the homotopy type of $\fD^\circ(n)$ and $\C(n)$, together with the proof of Propositions~\ref{FibreHomotopyTypeProposition} and \ref{ActionOnTheFibreProposition}, are given in \S\ref{LongExactSequenceSection}.

\skippy

Now let us discuss the proof of Theorem~C.
This theorem states that $\calE^\circ(n)$ is locally fibrewise homotopy equivalent to a product, or in other words, that it is a \emph{homotopy fibre bundle}.
Recall that $\calE^\circ(n)$ consists of maps from fibres of $\R\fD^\circ(n)$ to fibres of $\R\C(n)$ satisfying certain boundary conditions.
Our proof of Theorem~C will follow from general results on spaces of maps satisfying boundary conditions, together with a study of the local fibrewise structure of $\R\fD^\circ(n)$ and $\R\C(n)$.

In Section~\ref{MappingSpacesSection} we consider pairs $(X,f_X)$ and $(Y,f_Y)$, where $X$ and $Y$ are spaces and $f_X\colon A\to X$ and $f_Y\colon A\to Y$ are maps from some fixed space $A$.
We then introduce two mapping spaces, one consisting of maps $g\colon X\to Y$ \emph{relative to $A$}, or in other words satisfying $g\circ f_X=f_Y$, and a second consisting of maps $g\colon X\to Y$ \emph{homotopy-relative to $A$}, which are equipped with homotopies $g\circ f_X\simeq f_Y$.
We show that under certain conditions the two mapping spaces are homotopy equivalent, and that homotopy type of the space of homotopy-relative maps does not change when we replace $(X,f_X)$ or $(Y,f_Y)$ with a \emph{homotopy equivalent} pair $(Z,f_Z)$.

The program outlined above extends without modification to the fibrewise setting, where we can consider the pairs $(\R\fD^\circ(n),\partial)$ over $\fD^\circ(n)$ and $(\R\C(n),\partial)$ over $\C(n)$;
here $\partial$ denotes the combined boundary map.
The corresponding relative mapping space is $\calE^\circ(n)$, and it is homotopy equivalent to the homotopy-relative mapping space $\calE^\circ_h(n)$.

In Sections~\ref{RfDFibrewiseSection} and~\ref{RCFibrewiseSection} we study the pairs $(\R\fD^\circ(n),\partial)$ and $(\R\C(n),\partial)$ respectively.
We show that, locally over $\fD^\circ(n)$ and $\C(n)$, these pairs are homotopy equivalent to trivial fibred pairs of the form $(|a|\times \fD^\circ(n),\partial\times\mathrm{Id})$ and $(|c|\times\C(n),\partial\times\mathrm{Id})$ respectively.

Finally, in Section~\ref{QuasiFibrationProofSection} we combine the results of the preceding three sections to prove Theorem~C.

\section{Computing the long exact sequence}\label{LongExactSequenceSection}

This section is given to the proof of Propositions~\ref{FibreHomotopyTypeProposition} and \ref{ActionOnTheFibreProposition}.
Recall that these propositions compute the homotopy exact sequence associated to the quasifibration $\Pi\colon\calE^\circ(n)\to\fD^\circ(n)\times\C(n)$, and lead to the proof of Theorem~A.
In \S\ref{PureRibbonBraidGroupSubsection} and \S\ref{DiscsCactiPureRibbonBraidGroupSubsection} we recall the pure ribbon braid group and its relationship to framed discs and cacti.
Then in \S\ref{FibreSubsection} and \S\ref{ActionSubsection} we prove the propositions.

\subsection{The pure ribbon braid group}\label{PureRibbonBraidGroupSubsection}

This subsection collects some facts about the pure ribbon braid group that will be useful in the subsections to follow.

Recall, for example from~\cite{\KasselTuraev}, that E.~Artin's \emph{braid group $B_n$ on $n$ strands} is the group on generators $\sigma_1,\ldots,\sigma_{n-1}$ subject to the relations $\sigma_i\sigma_j=\sigma_j\sigma_i$ for $|i-j|\geqslant 2$ and $\sigma_i\sigma_{i+1}\sigma_i=\sigma_{i+1}\sigma_i\sigma_{i+1}$.
There is a homomorphism $B_n\to\Sigma_n$ sending $\sigma_i$ to the transposition of $i$ and $i+1$.
Its kernel is the \emph{pure braid group} $PB_n$, which has generators
\[\alpha_{ij}=\sigma_{j-1}\sigma_{j-2}\cdots\sigma_{i+1}\sigma_i^2\sigma_{i+1}^{-1}\cdots\sigma_{j-2}^{-1}\sigma_{j-1}^{-1}\]
for $1\leqslant i<j\leqslant n$.
The \emph{pure ribbon braid group} is the direct product $PRB_n=PB_n\times\mathbb{Z}^n$.
The generator of the $i$-th cyclic factor is written as $\zeta_i$.

Let $F_n$ denote the free group on generators $x_1,\ldots,x_n$.
There is an embedding $B_n\hookrightarrow\Aut(F_n)$ sending $\sigma_i$ to the transformation
\[x_i\mapsto x_ix_{i+1}x_i^{-1},\quad x_{i+1}\mapsto x_i,\quad x_l\mapsto x_l\quad\mathrm{for\ }l\neq i,i+1.\]
This embedding identifies $B_n$ with the subgroup consisting of those automorphisms $f\colon F_n\to F_n$ for which $f(x_1\cdots x_n)=x_1\cdots x_n$ and for which there is a permutation $\pi$ of $\{1,\ldots,n\}$ so that each $f(x_i)$ is conjugate to $x_{\pi(i)}$.
The same embedding identifies $PB_n$ with the analogous subgroup where the permutation $\pi$ is the identity map.

In later subsections we will use the following alternative description of the pure ribbon braid group, which is more directly related to the topological situation at hand.
Given $w=(w_1,\ldots,w_n)\in(F_n)^n$ let $\alpha(w)$ denote the homomorphism $F_n\to F_n$ that sends $x_i$ to $w_i x_i w_i^{-1}$.
Set
\[W_n=\{w\in(F_n)^n\mid \alpha(w)\in\Aut(F_n),\ \alpha(w)(x_1\cdots x_n)=x_1\cdots x_n\}.\]
There is a bijection $\Lambda\colon W_n \xrightarrow{\cong} PRB_n$.
The first component of $\Lambda(w)$ is $\alpha(w)\in PB_n$; the second component is $(m_1,\ldots,m_n)\in\mathbb{Z}^n$, where $m_i$ is the sum of the exponents of the letter $x_i$ in the word $w_i$.
The isomorphism $\Lambda$ translates the group structure on $PRB_n$ into the operation
\[(v_1,\ldots,v_n)\cdot(w_1,\ldots,w_n)=(\alpha(v)(w_1)v_1,\ldots,\alpha(v)(w_n)v_n).\]

\subsection{Discs, cacti and the pure ribbon braid group}\label{DiscsCactiPureRibbonBraidGroupSubsection}

This subsection recalls how both $\fD(n)$ (or equivalently $\fD^\circ(n)$) and $\C(n)$ are Eilenberg-MacLane spaces $K(PRB_n,1)$.
We begin by choosing basepoints.

\begin{definition}\label{anDefinition}
Let $a_n\in\fD^\circ(n)$ be an element whose little discs all have the same (small) radius, are arranged horizontally across the big disc in order $1,\ldots,n$, and which are embedded without rotation, so that the local marked point lies at the top of each little disc.
For $i=1,\ldots,n$ let $l_i$ denote the line segment in $|a_n|$ passing from the basepoint $\bullet$ of the big disc to the basepoint $\partial_i(\ast)$ of the $i$-th little disc.
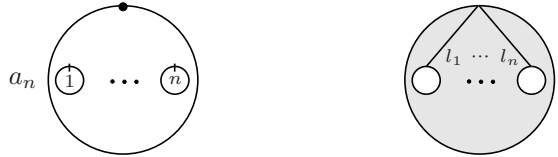
\begin{center}
\begin{lpic}[]{figures/BasepointDisc(,0.8 in)}
\lbl{-10,31.5;$a_n$}
\lbl{9.5,31.5;$\scriptstyle 1$}
\lbl{53.5,31.5;$\scriptstyle n$}
\lbl{170,41.5;$\scriptstyle l_1$}
\lbl{181.75,41.5;$\scriptstyle \cdots$}
\lbl{193.5,41.5;$\scriptstyle l_n$}
\end{lpic}
\end{center}
We will sometimes regard the $l_i$ as paths $l_i\colon[0,1]\to|a_n|$, and sometimes as subsets $l_i\subset |a_n|$.
The fundamental group $\pi_1(|a_n|,\bullet)$ can be identified with $F_n$; the $i$-th generator $x_i$ is represented by the loop which travels along $l_i$, then around the $i$-th circle and back down $l_i$.
\end{definition}

\begin{definition}\label{cnDefinition}
Let $c_n$ denote the cactus in which all lobes have equal length and meet at a unique point, and whose outgoing boundary $\partial_\rmout\colon S^1\to |c_n|$ has constant speed.
The intersection point coincides with the global marked point and also the local marked point on each lobe.
The lobes are cyclically ordered $n,\ldots,1$ at the intersection point, and the global marked point is chosen to lie on the $n$-th lobe.
Thus the outgoing boundary map $\partial_\rmout\colon S^1\to |c|$ traverses the lobes $n,n-1,\ldots,1$ in turn at constant speed.
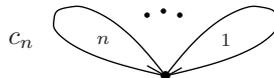
\begin{center}
\begin{lpic}[]{figures/BasepointCactus(,0.4 in)}
\lbl{-15,20;$c_n$}
\lbl{25,20;$\scriptstyle n$}
\lbl{85,20;$\scriptstyle 1$}
\end{lpic}
\end{center}
The fundamental group $\pi_1(|c_n|,\bullet)$ can be identified with the free group $F_n$ on generators $x_1,\ldots,x_n$.
Here $x_i$ is represented by the loop $\partial_i\colon S^1\to |c|$.
\end{definition}

It is well known that  $\fD(n)$ (or $\fD^\circ(n)$) is a $K(PRB_n,1)$.
For instance, $\fD(n)$ is homotopy equivalent to the product of $n$ circles with the configuration space $\mathrm{Conf}(n,\mathbb{R}^2)$ of $n$ ordered points in $\mathbb{R}^2$, and $\mathrm{Conf}(n,\mathbb{R}^2)$ is a $K(PB_n,1)$ (see section 1.4.1 of \cite{\KasselTuraev}, for example).
The generator $\alpha_{ij}$ corresponds to the loop based at $a_n$ in which the $i$-th and $j$-th little discs are brought side by side beneath the other little discs, and are then rotated around one another anticlockwise before returning to their original position.
Throughout this maneouver the framing of the little discs does not change.
The generator $\zeta_i$ corresponds to the loop in which the framing of the $i$-th little disc is rotated through one full turn.

A result of R.~Kaufmann~\cite[Proposition 3.3.19]{\Kaufmann} states that $\C(n)$ is a $K(PRB_n,1)$.
The generator $\alpha_{ij}$ corresponds to the loop based at $c_n$ in which the $j$-th lobe travels across the $(i+1),\ldots,(j-1)$-th lobes so that it is adjacent to the $i$-th lobe; the $i$-th and $j$-th lobes then rotate around one another anticlockwise before the $j$-th retraces its steps to its original position.
(This description of the generator is taken from the proof of Kaufmann's result quoted above.
In order to obtain this description we have reversed the path $\alpha_i$ in the proof there; this does not affect the outcome of that proposition, since the braid group admits an automorphism inverting all of its standard generators.)
The generator $\zeta_i$ corresponds to the loop in which the parameterization of the $i$-th little disc is rotated through one full turn.

\subsection{The fibre of $\Pi\colon \calE^\circ(n)\to\fD^\circ(n)\times\C(n)$}\label{FibreSubsection}

The purpose of this subsection is to prove Proposition~\ref{FibreHomotopyTypeProposition}, which describes the homotopy type of the fibre of $\Pi$.

\begin{definition}
Let $\mathcal{F}_n$ denote the fibre of $\Pi\colon \calE^\circ(n)\to\fD^\circ(n)\times\C(n)$ over $(a_n,c_n)$.
It is the space of homotopy equivalences $f\colon |a_n|\to|c_n|$ such that $f\circ\partial_i=\partial_i$ and $f\circ\partial_\rmout=\partial_\rmout$, equipped with the compact-open topology.
\end{definition}

\begin{definition}
Define a map $\mathcal{F}_n\to W_n$ as follows.
A point $f\in\mathcal{F}_n$ is sent to $(w_1,\ldots,w_n)$, where $w_i\in F_n=\pi_1(|c_n|,\bullet)$ is obtained by applying $f$ to the line segment $l_i$.
The boundary conditions guarantee that the image of $f\circ l_i$ is a loop in $|c_n|$.
\end{definition}

We will prove Proposition~\ref{FibreHomotopyTypeProposition} by showing that  this map $\mathcal{F}_n\to W_n$ is a homotopy equivalence.
The proof is given in the next three lemmas.

\begin{lemma}
The fibre $\mathcal{F}_n$ has weakly contractible components.
\end{lemma}
\begin{proof}
It will suffice to prove that, for any $f\in\mathcal{F}_n$, the loopspace $\Omega_f(\mathcal{F}_n)$ is contractible.
Note that $\Omega_f(\mathcal{F}_n)$ is the space of sections of the fibration $\Omega_f|c_n|\to |a_n|$ obtained by pulling back the free loopspace fibration $\Omega |c_n|\to L|c_n|\to |c_n|$ along $f$.
Its fibre over $x$ is $\Omega_{f(x)}|c_n|$.
We must show that the space of sections of this fibration is contractible.

Since $|c_n|$ is a $K(F_n,1)$, $\Omega_f|c_n|$ is fibre homotopy equivalent to a covering space with fibre $\pi_1(|c_n|,f(x))\cong F_n$ over $x$.
We must show that this covering space admits a unique section.
The covering space is determined by an action of the fundamental group $\pi_1(|a_n|,\bullet)$ of the base on the fibre $\pi_1(|c_n|,\bullet)$ over $\bullet$.
The sections of the covering space correspond to the fixed points of this action.

If $\phi\in\Aut(F_n)$ is the automorphism determined by $f$, then the action of $\pi_1(|a_n|,\bullet)=F_n$ on $\pi_1(|c_n|,\bullet)=F_n$ is given by $x_i(\omega)=\phi(x_i)\omega\phi(x_i)^{-1}$.
Since $\phi$ is an automorphism, $\omega$ is fixed by the action if and only if it commutes with every element of $F_n$.
The identity element is the unique element with this property, and this completes the proof.
\end{proof}

\begin{lemma}
The map $\mathcal{F}_n\to W_n$ is surjective.
\end{lemma}
\begin{proof}
Write $J_n\subset |a_n|$ for the union of the boundaries of the little discs with the line segments $l_i$.
Write $K_n\subset |a_n|$ for the union of $J_n$ with the boundary of the big disc $D^2$.
Choose $w=(w_1,\ldots,w_n)\in W_n$ and represent each $w_i$ by a based loop $\gamma_i\colon I\to |c_n|$.
Let $f\colon J_n\to |c_n|$ be the map that satisfies $f\circ\partial_i=\partial_i$ for each $i$ and that is given by $\gamma_i$ on $l_i$.
Now $J_n$ is a strong deformation retract of $|a_n|$, and so we may extend $f$ to a map $g\colon |a_n|\to |c_n|$.
This map satisfies $g\circ\partial_i=\partial_i$.
Also, since $w\in W_n$ it follows that $g$ is a homotopy equivalence, and that $g\circ\partial_\rmout$ is based homotopic to $\partial_\rmout$.
Then we can form a homotopy between $g|{K_n}$ and a map $K_n\to |c_n|$ that commutes with the $\partial_i$ and $\partial_\rmout$, and that is given by the $\gamma_i$ on $l_i$.
Since $K_n\hookrightarrow |a_n|$ is a cofibration, we may extend the homotopy to one between $g$ and some $h\colon |a_n|\to |c_n|$.
By construction $h$ is an element of $\mathcal{F}_n$ and its image in $W_n$ is exactly $w$.
\end{proof}

\begin{lemma}
Let $f_1,f_2\in\mathcal{F}_n$.
If the images of these elements in $W_n$ coincide, then they lie in the same component of $\mathcal{F}_n$.
\end{lemma}
\begin{proof}
Let $K_n\subset|a_n|$ be as in the preceding proof.
Now $f_1$ and $f_2$ coincide on the image of every boundary map and are homotopic rel.~endpoints on each $l_i$.
Since $K_n\hookrightarrow |a_n|$ is a cofibration, we may therefore apply a homotopy to $f_2$, through elements of $\mathcal{F}_n$, until it coincides with $f_1$ on all of $K_n$.
Let us therefore assume that $f_1$ and $f_2$ coincide on $K_n$.

We may now combine $f_1$ and $f_2$ to obtain a map $f\colon |a_n|\cup_{K_n} |a_n|\to |c_n|$.
Let $\overline{|a_n|}$ denote the space obtained from $|a_n|\times I$ by identifying $K_n\times I$ with a single copy of $K_n$.
It has `boundary' $|a_n|\cup_{K_n} |a_n|$.
To produce a homotopy between $f_1$ and $f_2$ it will suffice to extend $f$ to all of $\overline{|a_n|}$.

Choose an identification map $D^2\to |a_n|$ that is a bijection away from the boundary $S^1$.
Then we obtain identification maps $S^2\to |a_n|\cup_{K_n} |a_n|$ and $B^3\to\overline{|a_n|}$.
To extend $f$ to $\overline{|a_n|}$ it will suffice to extend the composite $S^2\to |a_n|\cup_{K_n}|a_n|\xrightarrow{f}|c_n|$ to $B^3$.
But since $|c_n|$ is a $K(F_n,1)$ this is always possible.
This completes the proof.
\end{proof}

\subsection{The action of $\pi_1(\fD^\circ(n)\times\C(n))$ on $\mathcal{F}_n$}\label{ActionSubsection}

This subsection will give the proof of Proposition~\ref{ActionOnTheFibreProposition}, which states that the fundamental group $PRB_n\times PRB_n$ of $\fD^\circ(n)\times\C(n)$ acts on $\pi_0(\mathcal{F}_n)\cong PRB_n$ by the rule $(\gamma,\delta)\cdot\phi=\delta\cdot\phi\cdot\gamma^{-1}$.
The proof consists of Lemma~\ref{PiOneActionLemmaOne},  where we show that  $(\gamma,\gamma)\cdot 1=1$, and Lemma~\ref{PiOneActionLemmaTwo}, where we show that $(\gamma,1)\cdot \phi=\phi\cdot\gamma^{-1}$.

In general, if $F\to E\to B$ is a fibration, then the action of $\pi_1(B)$ on $\pi_0(F)$ is defined as follows.
Take a loop $\gamma$ in $B$ and a point $f$ in $F$.
Then $\gamma$ can be lifted to a \emph{path} in $E$ with initial point $f\in F$ and final point $g\in F$, and the effect of $[\gamma]$ on $[f]$ is given by $[\gamma]\cdot[f]=[g]$.
If $F\to E\to B$ is only a quasifibration then the lifting process just described may not be possible.
However, when it is possible, then it still leads to a description of the action of $\pi_1(B)$ on $\pi_0(F)$.

Choose a basepoint $f_n\in\mathcal{F}_n$, where $f_n$ is any function sending each $l_i$ to the constant path.

\begin{lemma}\label{PiOneActionLemmaOne}
The action of Proposition~\ref{ActionOnTheFibreProposition} satisfies $(\gamma,\gamma)\cdot 1=1$ for all $\gamma\in PRB_n$.
\end{lemma}
\begin{proof}
It will suffice to prove this for $\gamma=\zeta_i$ and for $\gamma=\alpha_{ij}$.
Loops in $\fD^\circ(n)$ and $\C(n)$ representing these elements were described in \S\ref{DiscsCactiPureRibbonBraidGroupSubsection}.

Let us begin with the case $\gamma=\zeta_i$.
Let $t\mapsto a_n(t)$ and $t\mapsto c_n(t)$ be the loops in $\fD^\circ(n)$ and $\C(n)$ that represent $\zeta_i$.
Then $t\mapsto(a_n(t),c_n(t),f_n)$ determines a \emph{loop} in $\calE^\circ(n)$ and our claim follows.

Now fix $\gamma=\alpha_{ij}$ and let $t\mapsto a_n(t)$, $t\mapsto c_n(t)$ be the loops in $\fD^\circ(n)$ and $\C(n)$ respectively that represent $\alpha_{ij}$.
Then for each $t\in S^1$ we can find maps $f_n(t)\colon |a_n(t)|\to |c_n(t)|$ such that $t\mapsto (a_n(t),c_n(t),f_n(t))$ defines a \emph{loop} in $\calE^\circ(n)$.
For at all times $t\in S^1$ we can find an oriented embedding of $|c_n(t)|$ into $|a_n(t)|$ that sends the basepoint to the basepoint, such that the image of the $i$-th lobe separates the $i$-th little disc from the others, and such that the highest point each lobe is its local marked point.
Then the required $f_n(t)$ can be constructed by projecting out from the centre of the $i$-th little disc onto the $i$-th lobe, and inwards from the boundary of the big disc onto the cactus.

For example take $n=2$, $i=1$, $j=2$.
Then $a_n(t)$ and $c_n(t)$ are depicted here:
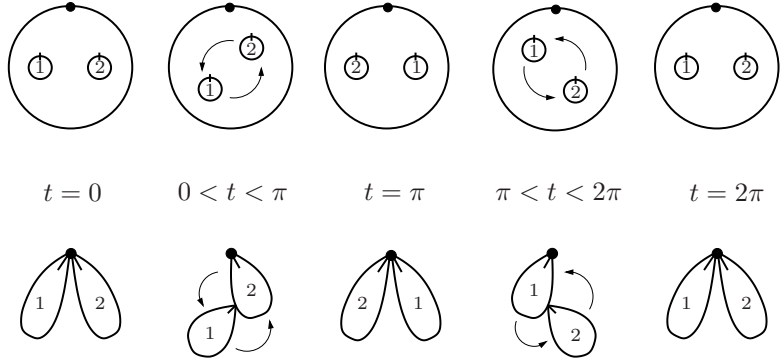
\begin{center}
\begin{lpic}[]{figures/DiscCactiLoops(4 in,)}
\lbl{16.5,148.5;$\scriptstyle 1$}
\lbl{47,148.5;$\scriptstyle 2$}
\lbl{102.5,137;$\scriptstyle 1$}
\lbl{123.5,158;$\scriptstyle 2$}
\lbl{207,148.5;$\scriptstyle 1$}
\lbl{176.5,148.5;$\scriptstyle 2$}
\lbl{266.5,157;$\scriptstyle 1$}
\lbl{287.5,136;$\scriptstyle 2$}
\lbl{343.5,148.5;$\scriptstyle 1$}
\lbl{374,148.5;$\scriptstyle 2$}
\lbl{16.5,28.5;$\scriptstyle 1$}
\lbl{47,28.5;$\scriptstyle 2$}
\lbl{102.5,13;$\scriptstyle 1$}
\lbl{123.5,35;$\scriptstyle 2$}
\lbl{208,28.5;$\scriptstyle 1$}
\lbl{179.5,28.5;$\scriptstyle 2$}
\lbl{266.5,35;$\scriptstyle 1$}
\lbl{287.5,12;$\scriptstyle 2$}
\lbl{343.5,28.5;$\scriptstyle 1$}
\lbl{374,28.5;$\scriptstyle 2$}
\lbl{33,85;$t=0$}
\lbl{114,85;$0<t<\pi$}
\lbl{196,85;$t=\pi$}
\lbl{279,85;$\pi<t<2\pi$}
\lbl{364,85;$t=2\pi$}
\end{lpic}
\end{center}
And the family of embeddings $|c_n(t)|\to|a_n(t)|$ is depicted here:
\begin{center}
\begin{lpic}[]{figures/EmbeddedCactiLoops(4 in,)}
\end{lpic}
\end{center}
\end{proof}

\begin{lemma}\label{PiOneActionLemmaTwo}
The action of Proposition~\ref{ActionOnTheFibreProposition} satisfies $(\gamma,1)\cdot \phi=\phi\cdot\gamma^{-1}$ for all $\gamma\in PRB_n$.
\end{lemma}

\begin{proof}
We identify $PRB_n$ with $W_n$.
Represent $\phi\in W_n$ by $(a_n,c_n,f)\in\mathcal{F}_n$.
Represent $\gamma$ by a loop $t\mapsto a_n(t)$ in $\fD^\circ(n)$ based at $a_n$, so that $(\gamma,1)$ is represented by $t\mapsto (a_n(t),c_n)$.
The representative of $\gamma$ can be chosen so that there is a continuous family of homeomorphisms $h_n(t)\colon |a_n|\to |a_n(t)|$, $t\in[0,1]$, satisfying $h_n(t)\circ\partial_i=\partial_i$ and $h_n(t)\circ\partial_\rmout=\partial_\rmout$ for all $t$.
Then we can form a path in $\calE^\circ(n)$ given by $t\mapsto (a_n(t),c_n,f\circ h_n(t)^{-1})$.
The endpoint of this path is $(a_n,c_n,f\circ h_n(1)^{-1})\in\mathcal{F}_n$.

We must show that $f\circ h_n(1)^{-1}\in\calF_n$ represents $\phi\cdot\gamma^{-1}$.
To $h_n(1)$ we assign an element $w=(w_1,\ldots,w_n)\in W_n$ as follows.
The arc $h_n(1)(l_i)$ in $|a_n|$ has the same endpoints as $l_i$, and so is homotopic rel.~endpoints to the concatenation of a loop $\omega_i$ based at $\bullet$ with the arc $l_i$.
Then $w_i=[\omega_i]$.
Using the description of the group operation on $W_n$, it follows that $f\circ h_n(1)^{-1}$ represents precisely $\phi\cdot w^{-1}$.

So we must show that $w=\gamma$.
In fact the assignment $h_n(1)\mapsto w$ is a homomorphism from the group of boundary-fixing homeomorphisms of $|a_n|$ into $W_n$, and so it suffices to prove $w=\gamma$ in the special cases $\gamma=\alpha_{ij}$ and $\gamma=\zeta_i$.
In these cases the family of homeomorphisms $h_n(t)$ can be seen explicity, and the identities $w=\gamma$ follow.
\end{proof}

\section{Relative mapping spaces}\label{MappingSpacesSection}

This section begins our work towards a proof of Theorem~C.
Recall that Theorem~C describes the fibrewise homotopy type of $\calE^\circ(n)$ locally over $\fD^\circ(n)\times\C(n)$.
Recall also that $\calE^\circ(n)$ is a fibrewise mapping space, whose fibre over $(a,c)$ consists of homotopy equivalences $f\colon |a|\to |c|$ with the property that $f\circ\partial = \partial$.
Here we have written $\partial$ for the maps 
\begin{align*}
\partial_\rmout\sqcup\partial_1\sqcup\cdots\sqcup\partial_n\colon &\textstyle\bigsqcup_{i=0}^n S^1\to |a|,\\
\partial_\rmout\sqcup\partial_1\sqcup\cdots\sqcup\partial_n\colon &\textstyle\bigsqcup_{i=0}^n S^1\to |c|.
\end{align*}
In order to understand $\calE^\circ(n)$, this section will study spaces of maps satisfying a condition of the form $f\circ\partial = \partial$.

Fix a compact Hausdorff space $A$.
We will consider \emph{spaces under $A$}, by which we mean pairs $(X,f_X)$ consisting of a compact Hausdorff space $X$ and a continuous map $f_X\colon A\to X$.
There are no further assumptions on $f_X$.
Given spaces $(X,f_X)$ and $(Y,f_Y)$ under $A$, we will consider spaces of maps $g\colon X\to Y$ satisfying the strict condition $g\circ f_X=f_Y$, which we call \emph{maps relative to $A$}, and also maps $g\colon X\to Y$ equipped with a homotopy  $g\circ f_X\simeq f_Y$, which we call \emph{maps homotopy-relative to $A$}.

The section has two parts.
The first part introduces what it means for two spaces under $A$ to be \emph{homotopy equivalent}.
The second part introduces two mapping spaces between spaces under $A$, one space of maps relative to $A$, and another space of maps homotopy-relative to $A$.
We compare the two notions, and we show that the homotopy type of the homotopy-relative mapping space is preserved if one of the spaces under $A$ is replaced by a homotopy-equivalent space under $A$.

\begin{note}
Here we will prove non-fibrewise versions of all results.
However, all results immediately generalize to the setting of spaces fibred over $B$ by replacing `compact', `Hausdorff' and `compact-open' with their fibrewise analogues.
\end{note}

\begin{notation}
Homotopies will always be parameterized by $[0,1]=I$, and a homotopy $H$ from $f$ to $g$ will be written $H\colon f\Rightarrow g$.
Given homotopies $H_1\colon g_0\Rightarrow g_1$ and $H_2\colon g_1\Rightarrow g_2$, we write $H_2\cdot H_1\colon g_0\Rightarrow g_2$ for their concatenation.
Given maps
\[A\xrightarrow{k} B\xrightarrow{g,h} C\xrightarrow{l}D\]
and a homotopy $H\colon g\Rightarrow h$, we write $H\circ k\colon g\circ k\Rightarrow h\circ k$  and $l\circ H\colon l\circ g\Rightarrow l\circ h$ and for the induced homotopies between the composites.
\end{notation}

\subsection{Homotopy equivalences of spaces under $A$}

In this subsection we will study the following notion of homotopy equivalence between spaces under $A$.
It is somewhat nonstandard in the sense that, given spaces $(X,f_X)$ and $(Y,f_Y)$ under $A$, we will consider maps $g\colon X\to Y$ that do not necessarily satisfy $g\circ f_X=f_Y$, but only satisfy a homotopy version of the condition.
Nevertheless this is the notion best suited to our purposes.

\begin{definition}\label{HomotopyEquivalenceDefinition}
A \emph{homotopy equivalence} from $(X,f_X)$ to $(Y,f_{Y})$ is a 6-tuple $(\phi,\psi,G,H,K,L)$ consisting of maps
\[\phi\colon X\to Y,\qquad\psi\colon Y\to X,\]
homotopies
\[G\colon\psi\circ\phi\Rightarrow\mathrm{Id}_X,\qquad H\colon\phi\circ\psi\Rightarrow\mathrm{Id}_{ Y},\]
and homotopies
\[K\colon\phi\circ f_X\Rightarrow f_{ Y},\quad L\colon\psi\circ f_{ Y}\Rightarrow f_X\]
such that the two homotopies $L\cdot (\psi\circ K)$ and $G\circ f_X$ from $\psi\circ\phi\circ f_X$ to $f_X$ are homotopic relative to their endpoints, and such that the two homotopies
$K\cdot (\phi\circ L)$ and $H\circ f_{Y}$ from $\phi\circ\psi\circ f_{Y}$ to $f_{Y}$ are homotopic relative to their endpoints.
We write $(X,f_X)\simeq(Y,f_Y)$ if there is a homotopy equivalence from $(X,f_X)$ to $(Y,f_Y)$.  See Lemma~\ref{EquivalenceRelationLemma} below.
\end{definition}

\begin{example}
\begin{enumerate}
\item 
Let $(X,f_X)$ and $(Y,f_Y)$ be spaces under $A$.
If $X$ and $Y$ are homotopy equivalent relative to $A$, then $(X,f_X)$ and $(Y,f_Y)$ are homotopy equivalent in the sense above.
\item
If $(X,f_X)$ and $(X,g_X)$ are spaces under $A$ with $f_X\simeq g_X$, then $(X,f_X)$ and $(X,g_X)$ are homotopy equivalent.
\item
For any $c_1,c_2\in\C(n)$ the pairs $(|c_1|,\partial)$ and $(|c_2|,\partial)$ are homotopy equivalent in the present sense, even though there are no maps between them relative to $A= S^1\sqcup\bigsqcup_{i=1}^n S^1$.
See Section~\ref{RCFibrewiseSection}.
\end{enumerate}
\end{example}

\begin{lemma}\label{EquivalenceRelationLemma}
Homotopy equivalence is an equivalence relation on spaces under $A$.
\end{lemma}
\begin{proof}
The relation is evidently reflexive and symmetric.
If $(\phi,\psi,G,H,K,L)$ is a homotopy equivalence from $(X,f_X)$ to $(Y,f_Y)$ and $(\phi',\psi',G',H',K',L')$ is a homotopy equivalence from $(Y,f_Y)$ to $(Z,f_Z)$, then 
\[\left(\phi'\phi,\psi\psi',G\cdot (\psi G'\phi),H'\cdot(\phi' H\psi'),K'\cdot (\phi' K),L\cdot (\psi L')\right)\]
is a homotopy equivalence from $(X,f_X)$ to $(Z,f_Z)$.
Thus the relation is transitive.
\end{proof}

\begin{lemma}\label{RelativeHomotopyEquivalenceLemma}
Suppose given a space $(X,f_X)$ under $A$ and a cofibrant embedding $i\colon T\hookrightarrow X$ of a contractible space.
Let $(X/T,f_{X/T})$ be the space under $A$ in which $f_{X/T}$ is the composite of $f_X$ with $X\to X/T$. 
Then $(X,f_X)$ and $(X/T,f_{X/T})$ are homotopy equivalent.
\end{lemma}

\begin{proof}
Let $f\colon I\times X\to X$ extend the identity map $\{0\}\times X\to X$ and a null homotopy $I\times T\to T$.
Let $\phi\colon X\to X/T$ be the collapse map and let $\psi\colon X/T\to X$ be induced by $f_1=f|\{1\}\times X$.
Let $G$ be the homotopy given by $f$.
Let $H$ be the homotopy whose composite with $\phi$ is given by $f$; this $H$ exists because $f(I\times T)\subset T$.
Let $K$ to be the constant homotopy and let $L$ be the homotopy obtained from $G$.
Then $(\phi,\psi,G,H,K,L)$ is the required homotopy equivalence.
\end{proof}

\subsection{Spaces of maps between spaces under $A$}

Now we introduce two mapping spaces between spaces under $A$.

\begin{definition}
Let $(X,f_X)$ and $(Y,f_Y)$ be spaces under $A$.
Set
\[\Map_f(X,Y)=\{g\colon X\to Y \mid g\circ f_X=f_Y\}\]
and
\[\Map_f^h(X,Y)=\{(g,H)\mid g\colon X\to Y,\ H\colon g\circ f_X\Rightarrow f_Y\}.\]
These are topologized as subspaces of $\Map(X,Y)$ and $\Map(X\cup_{f_X}A\times I,Y)$ respectively, both equipped with the compact-open topology.
There is an inclusion $\Map_f(X,Y)\hookrightarrow\Map_f^h(X,Y)$ given by taking the constant homotopy.
\end{definition}

\begin{proposition}\label{StrictToHomotopyProposition}
The inclusion $\Map_f(X,Y)\hookrightarrow \Map_f^h(X,Y)$ is a homotopy equivalence so long as for any space $K$ the map
$f_X\times \mathrm{Id}_K\colon A\times K\to X\times K$ has the homotopy extension property.
This restricts to a homotopy equivalence between the subspaces consisting of those elements for which the map $g$ is a homotopy equivalence.
\end{proposition}

\begin{proof}
This is a standard result, but we give a proof here in order to indicate just how the given assumptions are used.

Since $X$ and $Y$ are compact Hausdorff, continuous maps $(X\cup_{f_X}A\times I)\times K\to Y$ are in bijection with continuous maps $K\to\Map((X\cup_{f_X}A\times I),Y)$.
Thus the evaluation map $\alpha\colon (X\cup_{f_X}A\times I)\times\Map_f^h(X,Y)\to Y$ is continuous and extends to a continuous map $\beta\colon X\times I\times\Map_f^h(X,Y)\to Y$.
By choosing an appropriate retraction we can form a map $\gamma\colon (A\times I)\times I\times\Map_f^h(X,Y)\to Y$ such that
\begin{eqnarray*}
\gamma((a,t),0,m)&=&\alpha(a,t,m), \\
\gamma((a,0),s,m)&=&\alpha(a,s,m),\\ 
\gamma((a,t),1,m)&=&\alpha(a,1,m),
\end{eqnarray*}
and such that if $m\in\Map_f(X,Y)$ then $\gamma((a,t),s,m)=m(a)$ is independent of $t$ and $s$.
Now $\beta$ and $\gamma$ combine to give a map $\omega\colon(X\cup_{f_X} A\times I)\times I\times\Map_f^h(X,Y)\to Y$ that has an adjoint $\phi\colon I\times\Map_f^h(X,Y)\to\Map_f^h(X,Y)$.

By construction, $\phi|{\{0\}\times\Map_f^h(X,Y)}$ is the identity;
$\phi|{\{1\}\times\Map_f^h(X,Y)}$ has image contained in $\Map_f(X,Y)$;
and $\phi|{I\times\Map_f(X,Y)}$ has image contained in $\Map_f(X,Y)$.
Thus $\phi|{\{1\}\times\Map_f^h(X,Y)}$ is the required homotopy inverse.
\end{proof}

The next result shows that the homotopy type of $\Map_f^h(X,Y)$ depends only on the homotopy equivalence class of $(X,f_X)$ and $(Y,f_Y)$.

\begin{proposition}\label{MappingSpacesFunctorialProposition}
Let $(\phi,\psi,G,H,K,L)$ be a homotopy equivalence from $(Y,f_Y)$ to $(\tilde Y,f_{\tilde Y})$.
Then the two maps
\[\Phi\colon\Map_f^h(X,Y)\to\Map_f^h(X,\tilde Y),\quad\Psi\colon\Map_f^h(X,\tilde Y)\to \Map_f^h(X,Y)\]
given by $\Phi(g,F)=(\phi\circ g,K\cdot(\phi\circ F))$ and $\Psi(\tilde g,\tilde F)=(\psi\circ\tilde g,L\cdot(\psi\circ\tilde F))$ are inverse homotopy equivalences.
This restricts to a homotopy equivalence between the subspaces consisting of homotopy equivalences.
The analogous result holds for homotopy equivalences from $(X,f_X)$ to $(\tilde X,f_{\tilde X})$.
\end{proposition}

\begin{proof}
We will show that $\Psi\circ\Phi\simeq\mathrm{Id}$.
For $\Psi\circ\Phi$ is the map
\[(g,F)\mapsto (\psi\circ\phi\circ g, L\cdot[(\psi\circ K)\cdot(\psi\circ\phi\circ F)]).\]
We have the following homotopies relative to endpoints:
\begin{eqnarray*}
L\cdot[(\psi\circ K)\cdot(\psi\circ\phi\circ F)] 
&\simeq& [L\cdot (\psi\circ K)]\cdot(\psi\circ\phi\circ F) \\
&\simeq& (G\circ f_Y)\cdot(\psi\circ\phi\circ F)\\
&\simeq& F\cdot (G\circ g\circ f_X).
\end{eqnarray*}
The first of these is a reparametrization, the second comes from our assumption on the two homotopies $\phi\circ\psi\circ f_Y\Rightarrow f_Y$, and the third is obtained by using $G$ and $F$ simultaneously.
It follows that $\Psi\circ\Phi$ is homotopic to
\[(g,F)\mapsto (\psi\circ\phi\circ g, F\cdot(G\circ g\circ f_X)).\]
This is homotopic to the identity, as required.
That $\Phi\circ\Psi$ is homotopic to the identity follows from the analogous argument, as does the second part of the proposition.
\end{proof}

\section{The fibrewise structure of $\R\fD^\circ(n)$}\label{RfDFibrewiseSection}

Let \[\partial\colon\bigsqcup_{i=0}^n S^1\times\fD^\circ(n)\to\R\fD^\circ(n)\]
denote the combined boundary map $\partial_\rmout\sqcup\bigsqcup_{i=1}^n\partial_i$.
In this section we will study the fibrewise properties of $\R\fD^\circ(n)$ over $\fD^\circ(n)$, relative to $\partial$, using the language of Section~\ref{MappingSpacesSection}.
To be precise, we will work in the setting of spaces fibred \emph{over} $\fD^\circ(n)$ (and its open subsets), and in this setting we will consider $(\R\fD^\circ(n),\partial)$ as a space \emph{under} $\bigsqcup_{i=0}^n S^1\times\fD^\circ(n)$.
We prove that each $a\in\fD^\circ(n)$ has a neighbourhood $U$ over which $(\R\fD^\circ(n),\partial)$ is isomorphic to $(|a|\times\fD^\circ(n),\partial\times\mathrm{Id})$.
Moreover, we prove that $(\R\fD^\circ(n),\partial)$ satisfies the hypothesis of Proposition~\ref{StrictToHomotopyProposition}.

\begin{proposition}\label{fDFibrewiseStructureProposition}
\begin{enumerate}
\item
The combined boundary map $\partial$ extends to a fibrewise open embedding $\bigsqcup_{i=0}^n S^1\times[0,1)\times\fD^\circ(n)\hookrightarrow\R\fD^\circ(n)$.
In particular it is a fibrewise cofibration.
\item
$\R\fD^\circ(n)\to\fD^\circ(n)$ is a fibre bundle.
Moreover, local trivializations $\R\fD^\circ(n)|U\cong U\times |a|$ can be chosen compatible with the boundary maps.
\end{enumerate}
\end{proposition}

\begin{proof}
We will construct a fibrewise open embedding $\bigsqcup_{i=0}^n S^1\times[0,1)\to\R\fD^\circ(n)$ extending $\partial$.
The first claim follows.

Find a continuous function $d\colon\fD^\circ(n)\to(0,1)$ such that, for each $a\in\fD^\circ(n)$, the little discs of $a$ can be dilated by a factor $(1+d(a))$ yet still be disjoint and lie within the interior of the disc of radius $(1-d(a))$.
Extend the $i$-th little disc of $a$ to an affine linear map $a_i\colon\mathbb{R}^2\to\mathbb{R}^2$.
The required embedding is then given in the fibre over $a$ by
$(\theta,r)\mapsto [1-d(a)r]\theta$
on the $0$-th cofactor and by
$(\theta,r)\mapsto a_i( [1+d(a)r]\theta)$
on the $i$-th cofactor, for $i>0$.

We now turn to the second claim, beginning with a special case.
Fix $a\in\fD^\circ(1)$.
Let $\phi\colon S^1\to S^1$ denote the framing of $a$.
There is a homeomorphism $S^1\times [0,1]\to |a|$ given by $(\theta,r)\mapsto (1-r)\partial_1(\theta) + r\phi(\theta)$.
In a neighbourhood $U$ of $a$ the map $(\theta,r)\mapsto (1-r)\partial_1(\theta)+r\phi(\theta))$ is a homeomorphism $S^1\times[0,1]\to |b|$.
By combining the two homeomorphisms we obtain a trivialization of $\R\fD^\circ(n)$ over $U$.

Let $a\in\fD^\circ(n)$.
Let $A\subset\fD^\circ(n)$ consist of those $b$ for which each little disc of $b$ lies in the interior of the corresponding little disc of $a$.
This is an open subset of $\fD^\circ(n)$, and such open subsets cover $\fD^\circ(n)$.
The required local trivialization of $\R\fD^\circ(n)$ near $a$ can now be formed by restricting attention to each little disc of $a$ and applying the result for the case $n=1$ given above.
\end{proof}

\section{The fibrewise structure of $\R\C(n)$}\label{RCFibrewiseSection}

The last section studied the fibrewise properties of realizations of framed little discs.
In this section we turn to the analogous issue for cacti.
Let \[\partial\colon\bigsqcup_{i=0}^n S^1\times\C(n)\to\R\C(n)\]
denote the combined boundary map $\partial_\rmout\sqcup\bigsqcup_{i=1}^n\partial_i$.
We will study the fibrewise properties of $\R\C(n)$ over $\C(n)$, relative to $\partial$, using the language of Section~\ref{MappingSpacesSection}.
To be precise, we will work in the setting of spaces fibred \emph{over} $\C(n)$ (and its open subsets), and in this setting we will consider $(\R\C(n),\partial)$ as a space \emph{under} $\bigsqcup_{i=0}^n S^1\times\C(n)$.
Our result for cacti is much more involved than its counterpart for framed little discs.

\begin{proposition}\label{CFibrewiseStructureProposition}
Fix $c\in\C(n)$ and form the fibrewise spaces $(\R\C(n),\partial)$ and $(|c|\times\C(n),\partial\times\mathrm{Id})$ under $\bigsqcup_{i=0}^n S^1\times\C(n)$.
Then there is an open neighbourhood $U$ of $|c|$ over which $(\R\C(n),\partial)$ and $(|c|\times\C(n),\partial\times\mathrm{Id})$ are homotopy equivalent.
\end{proposition}

\begin{corollary}
$\R\C(n)\to\C(n)$ is a quasifibration.
\end{corollary}
\begin{proof}
In Proposition~\ref{CFibrewiseStructureProposition} the local existence of the maps $\phi$ and $\psi$, and the homotopies $G$ and $H$, show that $\R\C(n)\to\C(n)$ is locally fibrewise homotopy equivalent to a product.
It is therefore locally a quasifibration.
Since maps that are locally quasifibrations are themselves quasifibrations, the claim follows.
\end{proof}

\subsection{Proof of Proposition~\ref{CFibrewiseStructureProposition}}

Fix $c\in\C(n)$.
For each $r\in\{1,\ldots,n\}$ choose closed intervals $I_r\subset \bar I_r \subset S^1$ such that $I_r$ lies in the interior of $\bar I_r$, such that $\partial_r(\bar I_r)$ does not meet the image of any incoming boundary maps besides $\partial_r$, and such that $\partial_\rmout(\bullet)$ does not lie in any $\bar I_r$.
Since the lobe coordinates of a cactus depend continuously on the cactus, we may choose a neighbourhood $U$ of $c$
over which the above properties of $I_r\subset \bar I_r$ still hold.
Set $P=\bigsqcup_r \partial_r(S^1\setminus\mathrm{int}(I_r))\subset\R\C(n)|U$.

Thus $P$ is a fibrewise subspace of $\R\C(n)|U$ whose fibre over $d$ is a copy of $|d|$ with an interval removed from each lobe; in other words, each fibre is a tree.
The {fibrewise} quotient $\R\C(n)/P$ is given in each fibre by the bouquet of circles $(\bigsqcup I_r)/(\bigsqcup \partial I_r)$.
It is therefore natural to expect that $\R\C(n)|U \to (\R\C(n)|U)/ P$ is a fibrewise homotopy equivalence with the product $(\bigsqcup I_r)/(\bigsqcup \partial I_r)\times U$.
Our proof of Proposition~\ref{CFibrewiseStructureProposition} will proceed along these lines.

Form the fibrewise spaces $\left( (\R\C(n)|U)/P,\partial_1\right)$ and $\left( (|c|/P_c)\times U,\partial_2 \right)$ under $\bigsqcup_{i=0}^n S^1\times U$, where $\partial_1$ and $\partial_2$ are the composites
\begin{gather*}
\partial_1\colon \textstyle\bigsqcup_{i=0}^n S^1\times U\xrightarrow{\partial}\R\C(n)|U\to (\R\C(n)|U)/P,\\
\partial_2\colon \textstyle\bigsqcup_{i=0}^n S^1\times U\xrightarrow{\partial\times\mathrm{Id}} |c|\times U\to (|c|/P_c)\times U.
\end{gather*}
We will construct three homotopy equivalences,
\begin{equation}\label{HomotopyEquivalencesEquation}
\left( \R\C(n)|U,\partial\right)
\simeq
\left( (\R\C(n)|U)/P,\partial_1\right)
\simeq
\left( (|c|/P_c)\times U,\partial_2 \right)
\simeq
\left( |c|\times U,\partial\right).
\end{equation}
Since homotopy equivalence is a transitive relation, this will suffice to prove the proposition.

There are fibrewise pushout squares
\begin{equation}\label{LocalFibrewisePushoutSquare}
\xymatrix{
\bigsqcup_r U\times \partial I_r\ar[r]\ar[d] & P\ar[d] \\
\bigsqcup_r U\times I_r\ar[r] & \R\C(n)|U,
}
\qquad\qquad
\xymatrix{
\bigsqcup_r U\times \partial I_r\ar[r]\ar[d] & P_c\times U\ar[d] \\
\bigsqcup_r U\times I_r\ar[r] & |c|\times U.
}
\end{equation}
In both cases the horizontal maps are the restrictions of $\bigsqcup_{r=1}^n\partial_r$ and the right-hand map is the inclusion.
The left hand maps of these squares are cofibrations, so the same is true for the right hand maps.
We make the following claim:
\begin{enumerate}
\item \label{LocalLemmaTwo}
$P$ is fibrewise contractible.
\end{enumerate}
From the claim and from Lemma~\ref{RelativeHomotopyEquivalenceLemma} we obtain the first and last homotopy equivalence of \eqref{HomotopyEquivalencesEquation}.
It remains to find the middle homotopy equivalence.
The diagrams \eqref{LocalFibrewisePushoutSquare} also show that the fibrewise spaces $(\R\C(n)|U)/P$ and $(|c|/P_c)\times U$ are fibrewise isomorphic, and moreover the isomorphism is such that the triangle 
\[
\xymatrix{
{} & \bigsqcup_{i=1}^n S^1\times U \ar[dr]^{\partial_2|}  \ar[dl]_{\partial_1|} & {} \\
(\R\C(n)|U)/P\ar@{<->}[rr]_\cong   &  & (|c|/P_c)\times U 
}
\]
commutes.
The triangle says nothing about the value of $\partial_1$ and $\partial_2$ on the first cofactor of $\bigsqcup_{i=0}^n S^1\times U$.
Now we make a further claim.
\begin{enumerate}
\addtocounter{enumi}{1}
\item \label{LocalLemmaThree}
The composites
\begin{align*}
S^1\times U\hookrightarrow\textstyle\bigsqcup_{i=0}^n S^1\times U\xrightarrow{\partial_1}&(\R\C(n)|U)/P\cong (|c|/P_c)\times U, \\
S^1\times U\hookrightarrow\textstyle\bigsqcup_{i=0}^n S^1\times U\xrightarrow{\partial_2}&(|c|/P_c)\times U
\end{align*}
are homotopic; here the first map is the inclusion of the $0$-th cofactor.
\end{enumerate}
From this claim and from the commutative triangle it follows that $\left( (\R\C(n)|U)/P,\partial_1\right)$ is isomorphic to $\left( (|c|/P_c)\times U,\partial'_2 \right)$, where $\partial'_2$ is some map homotopic to $\partial_2$, so that $\left( (|c|/P_c)\times U,\partial'_2 \right)$ is itself homotopy equivalent to $\left( (|c|/P_c)\times U,\partial_2 \right)$.
Thus $\left( (\R\C(n)|U)/P,\partial_1\right)$ is homotopy equivalent to $\left( (|c|/P_c)\times U,\partial_2 \right)$, as required.

It remains to prove our claims.
The second claim states that two maps $S^1\times U\to (|c|/P_c)\times U$ are homotopic.
By construction, both are families (parameterized by $U$) of \emph{based} maps from $S^1$ to a bouquet of circles, lying in the same homotopy class.
The claim follows since the space of based maps $S^1\to\bigvee S^1$ in a fixed homotopy class is contractible.

The first claim states that $P$ is fibrewise contractible.
Note that the fibres of $P$ are connected.
For, given $d\in U$, any two points of $|d|$ can be joined by a path that does not pass through any of the $I_r$.
Let us assume without loss that each $I_r\subset S^1$ is equal to $[1/2,1]\subset S^1$.
Then the lobe coordinates are maps $U\to [0,1/2]^{n-1}$ and the $\partial_r|I_r$ are maps $[0,1/2]\to[0,1/2]^n$.
Rescaling $[0,1/2]$ to $[0,1]$ we see that the fibred space $P\to U$ is a \emph{configuration of orthogonal line segments} in the sense described in the next section.
The second claim then follows from Proposition~\ref{OrthogonalLineSegmentsProposition} below.

\subsection{Configurations of orthogonal line segments}

Let $A$ be a topological space.
Suppose we are given functions $\mathbf{x}^i\colon A\to\mathbb{R}^{n-1}$ for $i=1,\ldots,n$.
Write these as $\mathbf{x}^i=(x^i_1,\ldots,x^i_{i-1},x^i_{i+1},\ldots,x^i_n)$.
We obtain functions $\psi^i\colon A\times I\to\mathbb{R}^n$ given by
\[(a,t)\mapsto (x^i_1(a),\ldots,x^i_{i-1}(a),t,x^i_{i+1}(a),\ldots,x^i_n(a)).\]
Assume that for each $a\in A$ the space $\bigcup_i \psi^i(\{a\}\times I)$ is connected.
Define $Q\to A$ to be the fibrewise space $Q=\bigsqcup_i\psi^i(A\times I)\subset A\times\mathbb{R}^{n}$.
We refer to $Q$ as a \emph{configuration of orthogonal line segments} in $\mathbb{R}^n$.

\begin{proposition}\label{OrthogonalLineSegmentsProposition}
$Q$ is fibrewise contractible.
\end{proposition}
\begin{proof}

Since $Q$ obviously admits sections, it will suffice to show that it is fibrewise convex, or in other words that there is a fibrewise $h\colon Q\times_A Q\times [0,1]\to Q$ satisfying $h((p,q),0)=p$ and $h((p,q),1)=q$.

For the beginning of this proof let us assume that $A$ is a single point and omit it from the notation.

Note that the $i$-th coordinates of the points where $\psi^i[0,1]$ meets other $\psi^j[0,1]$ are exactly the $x^j_i$.

We say that $i\in\{1,\ldots,n\}$ is a \emph{leaf} if $\psi^i[0,1]$ meets exactly one other $\psi^j[0,1]$.
Leaves exist so long as $n\geqslant 2$.
This is clear when $n=2$, since $Q$ is connected.
We prove the general case by induction.
Let $\pi\colon\mathbb{R}^{n-1}\to\mathbb{R}^{n-2}$ be the map that forgets the final coordinate.
Then $\pi\mathbf{x}^1,\ldots,\pi\mathbf{x}^{n-1}$ define a configuration of line segments in $\mathbb{R}^{n-1}$, and so have a leaf which without loss is $(n-1)$.
Then the ${x}^1_{n-1},\ldots,{x}^{n-2}_{n-1}$ coincide.
So if $(n-1)$ is not a leaf then ${x}^n_{n-1}$ must differ from the common value of the $x^i_{n-1}$, and it follows that $n$ is a leaf.
In either case, a leaf exists.

An \emph{adapted path} $\gamma\colon[0,1]\to Q$ is one for which there is $v\geqslant 0$ and a decomposition of $[0,1]$ into intervals $[\alpha,\beta]$ on which $\gamma$ has the form $t\mapsto(x^i_1,\ldots,t_0+\epsilon_i v t,\ldots,x^i_n)$ for some value of $i$, some $t_0$, and some $\epsilon_i\in\{\pm 1\}$ depending only on $i$.

Any two points $x,y$ of $Q$ can be joined by a unique adapted path.
This is clear when $n=1$ and $n=2$, and in general is proved by induction.
For we can assume that $n$ is a leaf.
The claim is then immediate if $x$ and $y$ both lie in $\psi^n[0,1]$, and follows by induction if they both lie in $\bigcup_{i\leqslant n-1}\psi^i[0,1]$.
In the final case assume without loss that $x$ lies in $\psi^n[0,1]$ and that $y$ lies in $\bigcup_{i\leqslant n-1}\psi^i[0,1]$.
Write $x'$ for the point where $\psi^n[0,1]$ meets $\bigcup_{i\leqslant n-1}\psi^i[0,1]$.
Then an adapted path from $x'$ to $y$ exists by induction, and can easily be modified to produce an adapted path from $x$ to $y$.
Uniqueness is proved by a similar induction.

Now remove the restriction that $A$ is a single point.
By applying the results just obtained in each fibre, we obtain a function $h\colon Q\times_A Q\times [0,1]\to Q$ determined by the fact that each $h(x,y,-)$ is the adapted path from $x$ to $y$.
We must show that it is continuous.
For fixed $i\in\{1,\ldots,n\}$ the subset $A_i$ of $A$ on which $i$ is a leaf is closed, so it suffices to prove continuity over $A_i$.
Continuity over $A_i$ follows from an argument similar to the construction of adapted paths.
This completes the proof.
\end{proof}

\section{Proof of Theorem~C}\label{QuasiFibrationProofSection}

We can now complete the proof of Theorem~C.
We continue to write $\partial$ for the combined boundary map $\partial_\rmout\sqcup\bigsqcup_{i=1}^n\partial_i$.

\begin{definition}
Let $\calE^\circ_h(n)$ denote the space of quadruples $(a,c,f,H)$ where
$a\in\fD^\circ(n)$,
$c\in\C(n)$,
$f\colon |a|\to|c|$ is a homotopy equivalence,
and $H\colon\bigsqcup_{i=0}^n S^1\times[0,1]\to|c|$ is a homotopy from $f\circ\partial$ to $\partial$.
We topologize $\calE^\circ_h(n)$ as a subspace of the fibrewise mapping space.
Taking the constant homotopy gives an inclusion $\calE^\circ(n)\hookrightarrow\calE^\circ_h(n)$.
\end{definition}

Fix $a\in\fD(n)$ and $c\in \C(n)$.
Recall that $\Map_\partial(|a|,|c|)$ denotes the fibre of $\calE(n)$ over $(a,c)$.
In other words, it is the space of homotopy equivalences $f\colon |a|\to|c|$ satisfying $f\circ\partial=\partial$.

\begin{definition}
Given $a\in\fD^\circ(n)$ and $c\in\C(n)$, let $\Map^h_\partial(|a|,|c|)$ denote the fibre of $\calE^\circ(n)$ over $(a,c)$.
In other words, it is the space of pairs $(f,H)$ with $f\colon |a|\to|c|$ a homotopy equivalence and $H\colon [0,1]\times\bigsqcup_{i=0}^n S^1\to|c|$ a homotopy from $f\circ\partial$ to $\partial$.
It is equipped with the compact open topology.
Taking the constant homotopy gives an inclusion  $\Map_\partial(|a|,|c|)\hookrightarrow\Map_\partial^h(|a|,|c|)$.
\end{definition}

The proof of Theorem~A is completed by the following three propositions.

\begin{proposition}\label{ProofCompletionPropositionOne}
The inclusion $\calE^\circ(n)\hookrightarrow\calE^\circ_h(n)$ is a fibrewise homotopy-equivalence of spaces over $\fD^\circ(n)\times\C(n)$. 
\end{proposition}

\begin{proposition}\label{ProofCompletionPropositionTwo}
Every pair $(a,c)\in\fD^\circ(n)\times\C(n)$ has a neighbourhood $W$ over which there is a fibrewise homotopy equivalence
\[\calE^\circ_h(n)\simeq\Map^h_\partial(|a|,|c|)\times W.\]
\end{proposition}

\begin{proposition}\label{ProofCompletionPropositionThree}
The inclusion $\Map_\partial(|a|,|c|)\hookrightarrow\Map_\partial^h(|a|,|c|)$ is a homotopy equivalence.
\end{proposition}

\begin{proof}[Proof of Proposition~\ref{ProofCompletionPropositionOne}]
This is an application of the fibrewise version of Proposition~\ref{StrictToHomotopyProposition}; the assumption holds by the first part of Proposition~\ref{fDFibrewiseStructureProposition}.
\end{proof}

Proposition~\ref{ProofCompletionPropositionThree} is also proved by an application of Proposition~\ref{StrictToHomotopyProposition}.

\begin{proof}[Proof of Proposition~\ref{ProofCompletionPropositionTwo}]
Let $U$ be a neighbourhood of $a$ over which 
$(\R\fD^\circ(n),\partial)$ and $(|a|\times\fD^\circ(n),\partial\times\mathrm{Id})$ are homotopy equivalent as in Proposition~\ref{fDFibrewiseStructureProposition}, and
let $V$ be a neighbourhood of $c$ over which $(\R\C(n),\partial)$ and $(|c|\times\C(n),\partial\times\mathrm{Id})$ are homotopy equivalent, as in Proposition~\ref{CFibrewiseStructureProposition}.
Set $W=U\times V$.

Write $\calE'$ for the space of quadruples $((b,d),f,H)$, where $(b,d)\in W$, $f\colon |a|\to |c|$ is a homotopy equivalence, and $H\colon f\circ\partial\Rightarrow \partial$ is a homotopy.
Topologize $\calE'$ as a subspace of the fibrewise mapping space.
Proposition~\ref{MappingSpacesFunctorialProposition}, together with the previous paragraph, implies that $\calE^\circ_h(n)|W$ is  fibre homotopy equivalent to $\calE'$.
Part 8 of Proposition~\ref{FibrewiseMappingProposition} shows that $\calE'$ is fibrewise homeomorphic to the space given in the statement.
This completes the proof.
\end{proof}

\appendix

\section{Fibrewise topology}\label{FibrewiseAppendix}

This appendix will recall some basic notions of fibrewise topology, and establish some facts regarding the fibrewise mapping space.
We refer almost entirely to the book of James \cite{\James}.

\begin{definition}
Let $B$ be a space.
Then a \emph{fibred space} or \emph{fibrewise space} over $B$ is a space $X$ with a map $p\colon X\to B$ called the \emph{projection}.
The \emph{fibre} of $X$ over $b\in B$ is $X_b=p^{-1}(b)$.
A \emph{fibrewise map} $\phi\colon X\to Y$ between spaces fibred over $B$ is  a map $X\to Y$ that commutes with the projections.
\end{definition}

\begin{definition}
Given a map $f\colon A\to B$ and a space $X$ fibred over $B$, the pullback $f^\ast X$ is the fibred space $f^\ast X\to A$ given by $X\times_{p,f}A$ with its natural projection.
When $A\hookrightarrow B$ is the inclusion of a subspace we will write the pullback as $X_A$ or $X|_A$.
\end{definition}

\begin{definition}
Given $X,Y$ fibred over $B$, then $X\sqcup Y$ and $X\times_B Y$ become spaces fibred over $B$ in the obvious way.
The \emph{fibrewise pushout} of a diagram $Z\leftarrow X\rightarrow Y$ of spaces and maps fibred over $B$ is the ordinary pushout $Y\cup_X Z$, regarded as a space fibred over $B$.
A square
\[\xymatrix{
X\ar[r]\ar[d] & Y\ar[d]\\
Z\ar[r] & W
}\]
of spaces and maps fibred over $B$ is a \emph{fibrewise pushout square} if the induced fibrewise $Y\cup_X Z\to W$ is a homeomorphism.
\end{definition}

\begin{definition}
The fibrewise space $X$ is \emph{fibrewise compact} if $p\colon X\to B$ is proper, or in other words if $p$ is closed and has compact fibres.
It is \emph{fibrewise Hausdorff} if distinct points in the same fibre can be separated by open sets of $X$; this is always the case if $X$ itself is Hausdorff.
\end{definition}

\begin{proposition}[{\cite[3.7]{\James}}]\label{ContinuityProposition}
Let $X$ and $Y$ be fibrewise compact Hausdorff spaces.
A continuous bijection $\phi\colon X\to Y$ is in fact a homeomorphism.
If $\phi\colon X\to Y$ is a continuous fibrewise surjection, then a fibrewise function $\psi\colon Y\to Z$ is continuous if and only if the composite $\psi\circ\phi$ is continuous.
\end{proposition}

\begin{corollary}\label{FibrewisePushoutCorollary}
Suppose given a commutative square
\[\xymatrix{
X\ar[r]\ar[d] & Y\ar[d]\\
Z\ar[r] & W
}\]
of fibrewise compact Hausdorff spaces fibred over $B$, and suppose that in each fibre the square is a pushout.
Then the square is a fibrewise pushout.
Moreover, a fibrewise function $W\to V$ is continuous if and only if the composites $Y\to V$ and $Z\to V$ are continuous.
\end{corollary}

\begin{definition}[{\cite[p.63]{\James}}]
Let $X$ and $Y$ be spaces fibred over $B$.
The \emph{fibrewise mapping space} $\Map_B(X,Y)$ is the space of pairs $(b,f)$, where $b\in B$ and $f\colon X_b\to Y_b$ is continuous.
It is equipped with the topology generated by basic open sets $(K,V;W)$.
Here $W$ is an open subset of $B$, $K$ is a fibrewise compact subspace of $X_W$, and $V$ is an open subspace of $Y_W$.
Then $(K,V;W)$ consists of all pairs $(b,f)$ for which $b\in W$ and $f(K_b)\subset V_b$.
\end{definition}

\begin{proposition}[{\cite{\James}}]\label{FibrewiseMappingProposition}
Let $B$ be a topological space and let $X,Y,Z$ and $X_1,X_2$ be fibrewise compact Hausdorff spaces over $B$.

\begin{enumerate}
\item
Let $\theta\colon X\to Y$ and $\phi\colon Y\to Z$ be fibrewise maps.
Then there are fibrewise maps
\[\phi_\ast\colon\Map_B(X,Y)\to\Map_B(X,Z),\qquad \theta^\ast\colon\Map_B(Y,Z)\to\Map_B(X,Z).\]

\item
If $\theta$ above is a fibrewise surjection then $\theta^\ast$ is an embedding.

\item
The bijection
\[\Map_B(X_1\sqcup X_2,Y)\to\Map_B(X_1,Y)\times_B\Map_B(X_2,Y)\]
is a fibrewise homeomorphism.

\item
Take a map $F\colon A\to B$.
Then there is a continuous map
\[\Map_A(F^\ast X,F^\ast Y)\to\Map_B(X,Y)\]
sending $(a,f)$ to $(F(a),f)$.

\item
A fibrewise map
\[h\colon X\times_B Y\to Z\]
is continuous if and only if its adjoint
\[\hat h\colon X\to\Map_B(Y,Z)\]
is continuous.

\item
Let $C$ be a further space and regard $X\times C$, $Y\times C$ as fibred over $B\times C$.
Then the map
\[\Map_B(X,Y)\times C\to\Map_{B\times C}(X\times C,Y\times C)\]
sending $((b,f),c)$ to $((b,c),f)$ is continuous.

\item
Let $P$ and $Q$ be topological spaces, with $P$ compact Hausdorff.
Then the isomorphism of sets
\[\Map_B(B\times P,B\times Q)\cong B\times\Map(P,Q)\]
is a homeomorphism.
\end{enumerate}
\end{proposition}

\begin{proof}
\emph{All} references in this proof are to \cite{\James}.
Note that 3.22 and the comments after 3.12 guarantee that a fibrewise compact fibrewise Hausdorff space is also fibrewise locally compact and fibrewise regular.
Then parts 1, 2, 3, 5 and 6 are found in p.69, 9.4, 9.6, 9.7 and 9.13 respectively.
Part 4 is an immediate consequence of the definitions.

For part 6 take a standard open set $(K,V;W)$ in $\Map_{B\times C}(X\times C,Y\times C)$ and an element $((b,f),c)$ of its preimage.
We will find a neighbourhood of $((b,f),c)$ whose image is contained in $(K,V;W)$.
To do this we will reduce $(K,V;W)$ many times.
By reducing $W$ and $V$, assume that $W=W_1\times W_2$ and $V=V_1\times W_2$; this can be done because $K_{(b,c)}$ is compact.
Now using 3.14 and reducing $W_1$ if needed, find an open subset $J$ of $X_{W_1}$ such that $K_{(b,c)}\subset J$ and $\overline{J}_{b}\subset f^{-1}(V_1)$, where $\overline J$ is fibrewise compact over $W_1$.
Thus $K_{W_1\times W_2}\setminus J\times W_2$ is fibrewise compact over $W_1\times W_2$ and in particular its image in $W_1\times W_2$ is closed.
But this image does not contain $(b,c)$, and so by reducing $W_1$ and $W_2$ we may assume that $K_{W_1\times W_2}\subset J\times W_2\subset \overline J\times W_2$.
Now the open set $(\overline J, V_1;W_1)\times W_2$ contains $((b,f),c)$ and its image is contained in $(K,V;W)$ as required.

For the final part, continuity of $f\colon \Map_B(B\times P,B\times Q)\to B\times\Map(P,Q)$ follows from part 6, and continuity of its inverse follows from part 7.
\end{proof}

\end{document}